\documentclass[a4paper,11pt]{article}

\usepackage{amsmath,amsfonts,amssymb,amsthm}
\usepackage{times}
\usepackage{graphicx}
\usepackage{xcolor}
\usepackage{hyperref}
\hypersetup{
    colorlinks,
    linkcolor={red!50!black},
    citecolor={blue!50!black},
    urlcolor={blue!80!black}
}

\newcommand{\ind}{{\bf 1}}
\newcommand{\bZ}{\mathbb{Z}}
\newcommand{\bC}{\mathbb{C}}
\newcommand{\bN}{\mathbb{N}}

\newcommand{\bP}{\mathbb{P}}
\newcommand{\bE}{\mathbb{E}}
\newcommand{\bR}{\mathbb{R}}

\newtheorem{theorem}{Theorem}

\newtheorem{definition}{Definition}
\newtheorem{example}[theorem]{Example}\rm
\newtheorem{exmpl}[theorem]{Example} \rm

\newtheorem{lemma}[theorem]{Lemma}

\newtheorem{proposition}[theorem]{Proposition}
\newtheorem{remark}{Remark}\rm

\begin{document}

\title{\bf Generalized Coverage Processes  with Infinitely Divisible Finite Dimensional
Distributions}
\author{George Makatis and Michael A. Zazanis\thanks{Corresponding author.
 zazanis@aueb.gr} \\
Department of Statistics, \\
Athens University of Economics and Business,\\
Athens, Greece}
\date{}
\maketitle

\begin{abstract}
In this paper we define a class of coverage processes with
infinitely divisible finite dimensional distributions and a particular type of correlation
structure that can be thought of as generalizations of the classical
Ornstein--Uhlenbeck process and which include coverage processes
such as the $M/GI/\infty$ process.
We show how such processes arise naturally as limits of superpositions
of independent ON/OFF Markov processes with different parameters by
formulating an appropriate limit theorem. Various examples of processes of this
type are given.

\vspace{0.2in} \noindent {\small \textsc{Keywords:}}
\begin{minipage}[t]{4.5in}{\small \sc Superposition of ON/OFF processes,
Multivariate Probability Laws,
Ornstein--Uhlenbeck Process,
Coverage process, $M/GI/\infty$ process}
\end{minipage}
\end{abstract}

\section{Introduction}

We define a new family of processes which we term Generalized Coverage Processes.
These are stationary, infinitely divisible real valued processes defined on $\bR$ with
a special correlation structure similar to that of the classical coverage processes
such as the $M/G/\infty$ process (see Hall \cite{Hall}). Stationary infinitely divisible
processes have been examined by Maruyama \cite{Maruyama}. See also Lee \cite{Lee}
and Banrdorff-Nielsen \cite{BN} who examines stationary infinitely divisible
processes closer in spirit to those examined in this paper.

The classical Ornstein--Uhlenbeck process is the solution of the stochastic
differential equation
\begin{equation}  \label{OU0}
dX_t = -aX_t dt + dW_t
\end{equation}
where $\{W_t;t\geq 0\}$ is standard Brownian motion and $X_0$ a Gaussian
random variable independent of $\{W_t;t\geq 0\}$. If $a>0$ then the process
has a stationary version $\{X_t;t\geq 0\}$ when $X_0$ is Gaussian with zero
mean and variance equal to $\frac{1}{2a}$.

One particularly fruitful generalization of the classical
Ornstein--Uhlenbeck process consists in replacing the driving Brownian
Motion term in (\ref{OU0}) by a general L\'evy process.
The study of such processes and their applications in many areas, most
notably mathematical finance, has attracted a great deal of attention over
the last three decades. We refer the reader to Sato \cite{Sato} for background
on these Generalized Ornstein--Uhlenbeck (GOU) processes and to Barndorff--Nielsen
and Shephard \cite{BNS} for applications in finance and economics. We note
in particular that the class of GOU processes is intimately related to self--decomposable
distributions (see \cite{wolfe} and \cite{Jurek-Vervaat}).
In this paper we propose an alternative generalization of the stationary, Wiener process
driven, Ornstein--Uhlenbeck process which is intimately connected to coverage processes
and limits of superpositions of ON/OFF processes.

\section{The covariance structure of the classical Ornstein-Uhlenbeck process and an
algebraic identity}
The stationary version of the process (\ref{OU0}) is a zero mean Gaussian process
$\{X_t\}$ and the joint characteristic function of $(X_{t_1},X_{t_2},\ldots,X_{t_n})$ for
$t_1<t_2<\cdots <t_n$ is given by
\begin{equation}  \label{OU1}
\bE  \exp \left( i\sum_{i=1}^n \theta_i X_{t_i}\right) =\exp \left( -\frac{1}{4a}
\sum_{i=1}^n \theta_i^2\;-\;\frac{1}{2a} \sum_{1\leq i<j\leq n}\theta_i\theta_j e^{-a(t_j-t_i)}\right).
\end{equation}
(If we choose $a=1/2$ then the marginal distribution of the process is standard normal. Also,
here and throughout the paper the letter $i$ does double duty, both as the imaginary unit and
an index in the set of natural numbers, since no confusion is possible.)
The key observation, and point of departure for the generalization we have in mind, is to rewrite
the right hand side of (\ref{OU1}) in the form
\begin{eqnarray}  \label{key}
&&\exp \Bigl(-\frac{1}{4a}\sum_{1\leq i\leq j\leq n}\left( \theta_i + \cdots
+ \theta_j \right)^2 \left[ \left(1-e^{-a(t_j-t_{i-1})}\right) -\left(
1-e^{-a(t_j-t_i)}\right) \right.   \\
&&\hspace{2in}\left. -\left( 1-e^{-a(t_{j+1}-t_{i-1})}\right)
+\left(1-e^{-a(t_{j+1}-t_i)}\right) \right] \Bigr).  \notag
\end{eqnarray}
\textbf{In the above expression, {\em for any fixed $n$,} $t_0$ is to be interpreted as $-\infty $
whereas $t_{n+1}$ as $+\infty $.} This convention will be crucial in simplifying the notation and will
be used throughout the paper.
Also, note that the expression inside the square brackets can be rewritten, after elementary
manipulations, in the simpler form
\begin{equation}   \label{corelstructure}
\left(1-e^{-\mu (t_i-t_{i-1})}\right) \,e^{-\mu(t_j-t_i)}\, \left(1-e^{-\mu (t_{j+1}-t_j)}\right).
\end{equation}
The equality of the expressions (\ref{OU1}) and (\ref{key}) above 
is established in the following
\begin{proposition} \label{prop:key} Expression (\ref{key}) is equal to the
right hand side of (\ref{OU1}).
\end{proposition}
\begin{proof}
To show this we start with the following observation.
Let $\boldsymbol{\theta}^\top:=(\theta_1, \ldots , \theta_n)$ and let $B=[b_{ij}]$,
$i,j=1,\ldots,n$ be a symmetric $n\times n$ matrix. The quadratic form
$\boldsymbol{\theta}^\top B \boldsymbol{\theta}$ can always be written uniquely in the form
\begin{equation}  \label{quadratic}
\sum_{i=1}^n \sum_{j=1}^n b_{ij} \theta_i \theta_j
= \sum_{1\leq i\leq j\leq n}(\theta_i + \cdots + \theta_j)^2 a_{ij}
\end{equation}
where the $a_{ij}$ and $b_{kl}$ are uniquely determined from each other by
means of the relationships
\begin{equation}  \label{ba}
b_{kl}=\sum_{1\leq i\leq k\leq l\leq j\leq n}a_{ij},\;\;\; \text{ for }k\leq l
\end{equation}
and
\begin{equation}  \label{ab}
a_{ij} = b_{ij}-b_{i,j+1}-b_{i-1,j}+b_{i-1,j+1}, \hspace{0.1in} \text{ for }
1 \leq i \leq j \leq n,
\end{equation}
where in the last relationship we take $b_{0,j}= b_{i,n+1}=0$ for all $i,j$,
in $\{1,\ldots,n\}$.

Now use the identity (\ref{quadratic}) with $b_{ij} = e^{-a(t_j -t_i)}$, $1
\leq i \leq j \leq n$, and
\begin{eqnarray*}
a_{ij} &=& e^{-a(t_j-t_i)} - e^{-a(t_{j+1}-t_i)} - e^{-a(t_j-t_{i-1})} +
e^{-a(t_{j+1}-t_{i-1})} \\
&=& \left(1-e^{-a(t_j-t_{i-1})}\right) -\left( 1-e^{-a(t_j-t_i)}\right)
-\left( 1-e^{-a(t_{j+1}-t_{i-1})}\right)  +\left(1-e^{-a(t_{j+1}-t_i)}\right) .
\end{eqnarray*}
\end{proof}
Suppose now that $\phi(\theta)$ is the characteristic function of an infinitely divisible
probability law on $\bR$ and denote by $\psi(\theta):=\log \phi(\theta)$ its {\em characteristic
exponent} which has the form
\begin{equation} \label{psi_def}
 \psi(\theta)  = i \beta \theta - \frac{1}{2} \sigma^2 \theta^2
 + \int_{\bR\setminus \{0\}}
 \left(e^{i\theta x} - 1 -i\theta x \ind(|x|<1) \, \right) \nu(dx).
\end{equation}
Here $(\beta,\sigma, \nu)$ is the L\'evy triplet of the infinitely divisible law.
In particular the L\'evy measure $\nu$ is a $\sigma$-finite measure on $\bR\setminus\{0\}$
which satisfies the condition $\int_{0<|x|<1}x^2 \nu(dx) < \infty$
(see \cite{Applebaum}, \cite{Sato}).
In this paper we propose to replace the characteristic exponent $-\frac{1}{4a}\theta^2$
in the characteristic function in (\ref{key}) by the characteristic exponent $\psi(\theta)$
of a general infinitely divisible distribution (\ref{psi_def}),
and the exponential distribution function $1- e^{-at}$ which appears in (\ref{key}) by
a general {\em concave} distribution function $H$ on $\mathbb{R}^+$ as shown in
the sequel. In the following sections we will define this class of processes and we will
discuss some of its members.

\section{Infinitely divisible generalized coverage processes with a generalized
OU covariance structure}

In this section we define the class of {\em Generalized Coverage Processes
with Infinitely Divisible Finite Dimensional Distributions} (GCID process).
This class includes the $M/GI/\infty$ process, which may be thought of as
a process of intervals of random length whose left end points constitute a
Poisson process on the real line and in this sense a coverage process on the real
line, and the classic Ornstein-Uhlenbeck process which may arise also as a limit
of an array of coverage processes. This motivates the terminology.

Suppose that $f:[0,\infty )\rightarrow \mathbb{R}$ is {\em concave}, i.e.\
for all $x,y \in [0,\infty)$ and $\lambda \in [0,1]$
\[
f(\lambda x + (1-\lambda)y) \; \geq \; \lambda f(x) + (1-\lambda) f(y).
\]
Concave functions have the {\em decreasing increments} property, namely
\begin{equation} \label{decreasing_increments}
f(x+h) - f(x) \;\geq \; f(y+h) - f(y) \;\;\; \mbox{ for } 0\leq x \leq y, \;\;\;
h>0
\end{equation}
(see for instance \cite{Niculescu}). As a consequence of this property we can see
that the concave function $f$ is {\em supermodular} i.e.\ if
$0\leq x_1 <x_2 <x_3 <x_4$ the following inequality holds
\begin{equation}  \label{supermodular}
f(x_3 -x_1 )-f(x_3 -x_2 )-f(x_4-x_1 )+f(x_4-x_2 ) \;\geq \; 0.
\end{equation}
Indeed, to see that (\ref{decreasing_increments}) implies (\ref{supermodular})
it suffices to take $x=x_3-x_2$, $y=x_3-x_1$ and $h= x_4-x_3$.

The following Lemma will play a central role in the sequel.
\begin{lemma}    \label{le:multi_dimensional}
Assuming that $\phi(\theta)$ is the characteristic function of an infinitely divisible law on $\bR$
with characteristic exponent given by (\ref{psi_def}), for any $n \in \bN$ and nonnegative reals
$a_{ij}$, $1\leq i \leq j \leq n$, the function
\begin{equation} \label{phi_lemma2}
\phi_n(\theta_1,\ldots,\theta_n) :=
\exp \left( \sum_{1 \leq i \leq j \leq n} a_{ij} \psi(\theta_i+\cdots+\theta_j) \right)
\end{equation}
is the characteristic function of a probability measure on $\bR^n$. Furthermore, $\phi_n$ is
infinitely divisible and its L\'evy triplet,
$(\boldsymbol{\beta},\boldsymbol{\Sigma},\boldsymbol{\nu}_n)$, is
\begin{equation} \label{beta}
\boldsymbol{\beta}=(\boldsymbol{\beta}_1,\ldots,\boldsymbol{\beta}_n) \hspace{0.1in}
\mbox{ with } \hspace{0.1in}
\boldsymbol{\beta}_k = \beta\sum_{i=1}^k \sum_{j=k}^n a_{ij}, \;\; k=1,\ldots,n,
\end{equation}
\begin{equation} \label{CapSigma}
\boldsymbol{\Sigma}_{kl} :=   \sigma^2 \hspace{-0.1in}
\sum_{{1\leq i \leq (k\wedge l) }\atop { (k\vee l) \leq j \leq n}}
\hspace{-0.1in}  a_{ij},
\;\;\; k,l =1,\ldots,n.
\end{equation}
The L\'evy measure $\boldsymbol{\nu}_n$ on $\bR^n$ corresponding to $\phi_n$ can
be characterized as follows: Let $\mathbf{u}_{ij}$, $1\leq i \leq j \leq n$, be the collection
of vectors in $\bR^n$ such that $\mathbf{u}_{ij}=(0,\ldots,0,1,\ldots,1, 0,\ldots,0)^\top$
(i.e.\ its $k$th component is 1 if $i \leq k \leq j$ and 0 otherwise). Then, for any $0<r < R$
and any bounded, continuous function $\varphi:\bR \rightarrow \bR^n$ vanishing outside
of the set $\{\boldsymbol{x} :r<\Vert\boldsymbol{x}\Vert < R\}$,
\begin{equation} \label{nu}
\int_{\bR^n} \varphi(\boldsymbol{x}) \boldsymbol{\nu}_n(d\boldsymbol{x}) \;=\;
\sum_{1 \leq i \leq j \leq n} a_{ij} \int_{\bR} \varphi(s\mathbf{u}_{ij}) \nu(ds).
\end{equation}
\end{lemma}
\begin{proof}
Let $\{Z_{ij}(t); t \geq 0\}$, $1\leq i \leq j \leq n$, be independent L\'evy processes,
all with the same characteristic exponent $\psi(\theta)$ (and thus all with triplet $(\beta,\sigma^2,\nu)$).
Thus, for $\theta \in \bR$, $\bE e^{i \theta Z_{ij}(t)} =e^{ t\psi(\theta)}$, $t \geq 0$.
Define the random element $\mathbf{Y} =(Y_1,\ldots,Y_n)^\top$ of $\bR^n$ by
\begin{equation}  \label{Levy00}
\mathbf{Y} = \sum_{1\leq i \leq j \leq n}  Z_{ij}(a_{ij}) \, {\bf u}_{ij}.
\end{equation}
If ${\boldsymbol \theta} = \left(\theta_1,\ldots,
\theta_n\right)^\top$ then $\boldsymbol{\theta}^\top \mathbf{u}_{ij} = \theta_i+\cdots+\theta_j$
and in view of (\ref{Levy00}), $\mathbf{Y}$ has characteristic function
\begin{eqnarray}   \nonumber
\bE[ e^{i \boldsymbol{\theta}^\top \mathbf{Y} }] &=&
\bE\Big[ \exp\Big(i \sum_{1 \leq i \leq j \leq n}
\boldsymbol{\theta}^\top \mathbf{u}_{ij}  \,  Z_{ij}(a_{ij})  \Big)\Big]
\;=\; \bE\Big[ \exp\Big( i \sum_{1\leq i  \leq j \leq n} (\theta_i+\cdots+\theta_j) Z_{ij}(a_{ij})\Big) \Big]
 \\ \nonumber 
&=&  \prod_{1\leq i  \leq j \leq n} \bE e^{i  (\theta_i+\cdots+\theta_j) Z_{ij}(a_{ij})} \\ \label{Levy2}
&=& \prod_{1\leq i  \leq j \leq n} e^{a_{ij}\psi(\theta_i+\cdots+\theta_j)} \;=\;
\phi_n(\theta_1,\ldots,\theta_n)
\end{eqnarray}
where we have used the independence of the L\'evy processes
$\{Z_{ij}(t); t\geq 0\}$ and the fact that they all have the same characteristic exponent
$\psi(\theta)$. Equation (\ref{Levy2}) follows directly from (\ref{phi_lemma2}) and therefore
$\phi_n(\theta_1,\ldots,\theta_n)$ is
the characteristic function of $\mathbf{Y}$. Note also that $\mathbf{Y}$, as defined in
(\ref{Levy00}), is the sum of the
$n(n+1)/2$ independent, infinitely divisible random elements of $\bR^n$ of  (\ref{Levy00}), with
characteristic functions $\bE[e^{i \boldsymbol{\theta} \cdot \mathbf{u}_{ij} \, Z_{ij}(a_{ij})}]
= e^{\psi(\theta_i+\cdots+\theta_j) a_{ij}}$ and therefore $\mathbf{Y}$ is an infinitely
divisible random element of $\mathbb{R}^n$. The L\'evy triplet of $\phi_n$ can be obtained
from (\ref{psi_def}) with $\boldsymbol{\beta}=(\beta_1,\ldots,\beta_n)^\top$ given by
\[
\boldsymbol{\theta}^\top \boldsymbol{\beta} = \sum_{k=1}^n \theta_k \beta_k =
\sum_{1\leq i \leq j \leq n} a_{ij} \beta (\theta_i+\cdots + \theta_j) = \sum_{k=1}^n
\theta_k \beta \sum_{i=1}^k \sum_{j=k}^n a_{ij}
\]
thus establishing (\ref{beta}). Similarly
\[
\frac{1}{2} \boldsymbol{\theta}^\top \boldsymbol{\Sigma}\boldsymbol{\theta} =
\frac{1}{2} \sum_{i=1}^n \sum_{j=1}^n \Sigma_{ij}\theta_i \theta_j = \frac{1}{2}
\sum_{1\leq i \leq j \leq n} \sigma^2 a_{ij}(\theta_i+\ldots+\theta_j)^2.
\]
Then, from identiy (\ref{quadratic}) we obtain (\ref{CapSigma}).
Finally, the random variable
$\mathbf{u}_{ij} \, Z_{ij}(a_{ij})$ has L\'evy (and therefore probability) measure concentrated
on the one-dimensional linear subspace of $\mathbb{R}^n$
$V_{ij}:=\{\mathbf{x}: \mathbf{x} =  s\, \mathbf{u}_{ij} , \; s \in \mathbb{R} \}$.
This justifies (\ref{nu}) and completes the proof.
\end{proof}
\begin{theorem}[GCID]  \label{basic-theorem-H}
Suppose that $\phi (\theta )$ is the characteristic
function of an infinitely divisible law on $\bR$
and denote by $\psi(\theta):=\log \phi(\theta)$ its characteristic exponent given by
(\ref{psi_def}). Denote by $X$ a real random variable
with the corresponding distribution, i.e. $\phi (\theta )=\bE  e^{i\theta X}$.
Suppose also that $H$ is a probability distribution function on $\bR$
such that $H(0)=0$ and $H$ is concave. Then, for $n\in \mathbb{N}$ and
$\theta_i,t_i\in \mathbb{R}$, $i=1,2,\ldots ,n$, $t_1\leq t_2\leq \cdots
\leq t_n,$ the family of functions $\phi_n(\theta_1,\ldots ,\theta_n;
t_1,\ldots ,t_n)$ defined by
\begin{eqnarray}   \label{main}
&&\hspace{-0.2in} \log \phi_n(\theta_1,\ldots,\theta_n; t_1, \ldots, t_n)  \\
&&\hspace{0.3in} =\sum_{1\leq i\leq j\leq n}\psi (\theta_i+\cdots +\theta_j)\left[
H(t_j-t_{i-1})-H(t_j-t_i)-H(t_{j+1}-t_{i-1})+H(t_{j+1}-t_i)\right]  \notag
\end{eqnarray}
are characteristic functions of \emph{consistent, infinitely divisible} finite--dimensional
distributions. {\em For each fixed $n$ in the above expression $t_0$ is to be
interpreted as $-\infty $ and $t_{n+1}$ as $+\infty$.} The family of
characteristic functions defined in (\ref{main}) corresponds to the
finite--dimensional distributions of a stationary process $\{X_t;t\in
\mathbb{R}\}$ with marginal distribution $\bE  e^{i\theta X_t}=\phi (\theta )$. When
the random variable $X$ has finite second moment 
the process $\{X_t\}$ has covariance
\begin{equation}  \label{Cov}
\text{\textsf{Cov}}(X_t,X_{t+h})=-\psi^{\prime \prime }(0)\left[ 1-H(|h|) \right].
\end{equation}
\end{theorem}
\begin{remark}
\textrm{We point out that, for $n=1$, (\ref{main}) contains only one term, namely
$\log \phi_1(\theta_1;t_1)=\psi (\theta_1)$, while for $n=2$, after simple calculations
{\em taking into account the conventions regarding $t_0$ and $t_{n+1}$} we obtain
\begin{equation}  \label{property 1}
\log \phi_2(\theta_1,\theta_2;t_1,t_2)=\psi (\theta_1)H(t_2-t_1) +\psi
(\theta_1+\theta_2)\left[1-H(t_2-t_1)\right] +\psi (\theta_2)H(t_2-t_1).
\end{equation}
}
\end{remark}
\begin{proof}
Note that
\begin{equation}  \label{aij}
a_{ij}:=\left[ H(t_j-t_{i-1})-H(t_j-t_i)-H(t_{j+1}-t_{i-1})+H(t_{j+1}-t_i)\right] \geq 0
\end{equation}
since $H$ is concave and hence supermodular. 
Therefore, using Lemma \ref{le:multi_dimensional}, we conclude that the right hand side of 
(\ref{main}) is in fact the logarithm of a characteristic function
of an infinitely divisible probability law on $\mathbb{R}^n$.

There remains to show that this family of distributions satisfies the Kolmogorov consistency conditions.
Denote by $F(x_1,\ldots,x_n; t_1,\ldots,t_n)$ 
the distribution function corresponding to the characteristic function defined by (\ref{main}).
In order to establish for these distributions to satisfy the consistency conditions it suffices to show that,
for every $n\geq 2$ and every sequence of values $t_1< \cdots < t_k <\cdots <t_n$,
\begin{eqnarray} \label{consistent}
&& \left. \log
\phi_{n}(\theta_1,\ldots,\theta_{k-1},\theta_k,\theta_{k+1}, \ldots,\theta_n \:;\:
t_1,\ldots,t_{k-1},t_k,t_{k+1},\ldots, t_{n}) \right|_{\theta_k=0} \\  \nonumber
&& \hspace{1.1in} = \; \log \phi_{n-1}(\theta_1,\ldots,\theta_{k-1},\theta_{k+1},
\ldots,\theta_n \: ; \: t_1,\ldots,t_{k-1},t_{k+1},\ldots, t_{n}) .
\end{eqnarray}
(The case where $k=1$ or $k=n$ is similar and its discussion is omitted.)
In view of (\ref{main}) both sides of (\ref{consistent})
consist of terms of the form
\begin{equation}  \label{term_ij}
\psi(\theta_i+\cdots+\theta_j) a_{ij}, \;\;\; 1 \leq i \leq j \leq n,
\end{equation}
where $a_{ij}$ is given by (\ref{aij}). If $k \notin \{i,i+1,j-1,j\}$   then the term
(\ref{term_ij}) apears both in the right and left hand side of (\ref{consistent}). When
$k=i$ or $k= i+1$ on the left of (\ref{consistent}) we have the two terms
\begin{eqnarray}  \label{combine}
&& \hspace{-0.4in} 
\psi(\theta_k+\cdots+\theta_j)
\left[ H(t_j-t_{k-1})-H(t_j-t_k)-H(t_{j+1}-t_{k-1})+H(t_{j+1}-t_k)\right], \\  \nonumber
&& \hspace{-0.4in} 
\psi(\theta_{k+1}+\cdots+\theta_j)
\left[ H(t_j-t_{k})-H(t_j-t_{k+1})-H(t_{j+1}-t_{k})+H(t_{j+1}-t_{k+1})\right].
\end{eqnarray}
Setting $\theta_k =0$ gives $\psi(\theta_k+\cdots+\theta_j) =
\psi(\theta_{k+1}+\cdots+\theta_j)$. Hence the two terms in (\ref{combine})
are combined into one as follows
\begin{eqnarray}   \nonumber
&& \hspace{-0.35in} \psi(\theta_{k+1}+\cdots+ \theta_j)
[ H(t_j-t_{k-1})-H(t_j-t_k)-H(t_{j+1}-t_{k-1})+H(t_{j+1}-t_k) \\  \nonumber
&& \hspace{0.8in} + H(t_j-t_{k})-H(t_j-t_{k+1})-H(t_{j+1}-t_{k})+H(t_{j+1}-t_{k+1}) ]  \\ \nonumber
&& \hspace{-0.35in} = \psi(\theta_{k+1}+\cdots+ \theta_j) \\
&& \hspace{0.3in} \times [ H(t_j-t_{k-1})-H(t_{j+1}-t_{k-1})
 -H(t_j-t_{k+1})+H(t_{j+1}-t_{k+1}) ].  \label{combine2}
\end{eqnarray}
The expression in (\ref{combine2}) is equal to the corresponding term in the right
hand side of (\ref{consistent}). The same can be shown when $k=j-1$ and $k=j$ which
gives the term containing $\psi(\theta_i+\cdots+\theta_{k-1})$ and thus we establish
(\ref{consistent}).

The stationarity of the process
$\{X_t;t\in \mathbb{R}\}$ is immediate from the fact that
\begin{equation*}
\log \phi_n(\theta_1,\ldots ,\theta_n;t_1+\tau ,\ldots ,t_n+\tau )
=\log \phi_n(\theta_1,\ldots ,\theta_n;t_1,\ldots ,t_n),
\end{equation*}
for any $\tau \in \mathbb{R}$, as can be readily verified from (\ref{main}).

To establish (\ref{Cov}) use (\ref{property 1}) with $\theta_2=\theta$,
$\theta_1 = -\theta$ and $h=t_2-t_1$ in  to obtain
\begin{eqnarray} \nonumber
\phi_2(-\theta ,\theta ;t,t+h) &=&\bE  \left[ e^{i\theta \left(
X_{t+h}-X_t\right) }\right] \\
&=& \exp \left\{ \psi \left( 0\right) \left[ 1-H(h)
\right] +\psi (\theta )H(h)+\psi (-\theta )H(h)\right\}  \notag \\
&=&\exp \left\{ H(h)\left[ \psi (\theta )+\psi (-\theta )\right] \right\} \label{Difference}
\end{eqnarray}
where in the last equation we have taken into account that $\psi \left(0\right) =0$.
We have also assumed that $h>0$. The situation where $h$
is negative can be dealt with similarly. Differentiating twice this
expression and evaluating at $\theta =0$ we obtain
\begin{equation}  \label{second-moment}
\bE  (X_t-X_{t+h})^2=-\left. \frac{d^2}{d\theta^2}\phi_2(-\theta,\theta;t,t+h)\right\vert_{\theta=0}
=-2H(h)\psi^{\prime \prime}(0).
\end{equation}
From the stationarity of the process $X_t$ we have
$\mathsf{Cov}(X_t,X_{t+h})=\mathsf{Var}(X_t)-\frac{1}{2} \bE  (X_t-X_{t+h})^2$.
Taking into account the fact that
$\psi^{\prime \prime}(0) = \mathsf{Var}(X_t)$ and (\ref{second-moment})
establishes (\ref{Cov}). We also note that $H$ is necessarily continuous and therefore the process
is necessarily mean square continuous. 
\end{proof}

\begin{remark}
{\rm A special case of particular importance arises when $H(t)=1-e^{-\mu t}$
for some $\mu >0$. In that case
\begin{eqnarray}  \label{cor-str}
\log \phi_n(\theta_1,\ldots ,\theta_n;t_1,\ldots,t_n) && \\
&& \hspace{-1.3in} =\;\sum_{1\leq i\leq j\leq n}\psi
(\theta_i+\cdots +\theta_j) \, \left(1-e^{-\mu (t_i-t_{i-1})}\right) \,
e^{-\mu(t_j-t_i)}\, \left(1-e^{-\mu (t_{j+1}-t_j)}\right).   \nonumber
\end{eqnarray}
Note that the correlation structure in the above expression is of the form (\ref{corelstructure})
which corresponds to that of an Ornstein-Uhlenbeck process (\ref{key}).
In particular, if second moments exist, the covariance function is given by
$\mathsf{Cov}(X_t,X_{t+h}) = \mathsf{Var}(X_0)e^{-\mu \left\vert h\right\vert}$.}
\end{remark}

\begin{remark} \label{rem:gaussian}
{
\rm Suppose that $\{X_t; t\in \bR\}$ is a real-valued, zero-mean, {\em stationary} Gaussian process
with covariance function $C(t):= \bE[X_sX_{s+t}]$, $t \geq 0$, with $C(0)>0$ and
$\lim_{t\rightarrow \infty}C(t) = 0$. Define the correlation function $r(t):= C(t)/C(0)$, $t\geq0$.
If the covariance function $C$ is a {\em convex} function then the Gaussian process $\{X_t;t \in \bR\}$
belongs to the family GCID with $\psi(\theta) := -\frac{C(0)}{2} \theta^2$ and $H(t) := 1-r(t)$.
Then, with the standard assumption that $t_0=-\infty$ and $t_{n+1}=+\infty$, 
\begin{eqnarray} \nonumber
\bE[e^{i \sum_{i=1}^n \theta_i X_{t_i}}] &=&
e^{-\frac{1}{2}\sum_{i,j=1}^n \theta_i\theta_j C(|t_i-t_j|)} \\  \label{gaussian_exmpl}
&=& e^{\frac{C(0)}{2} \sum_{1\leq i \leq j \leq n} (\theta_i+\cdots+\theta_j)^2
\left[ H(t_{j+1}-t_{i-1}) - H(t_{j+1}-t_i) -H(t_j-t_{i-1}) +H(t_j-t_i)\right]}.
\end{eqnarray}
It is clear that the above equality holds for {\em any} stationary Gaussian process for which
$\lim_{t\rightarrow \infty}C(t) =0$ as a result of the identity (\ref{quadratic}). However
only when $r(t)$ is convex (and therefore $H$ concave) do we have $ H(t_{j+1}-t_{i-1}) - H(t_{j+1}-t_i)
 -H(t_j-t_{i-1}) +H(t_j-t_i) \geq0$.
}
\end{remark}

\begin{lemma}  \label{le:drift}
For any probability distribution $H$ on $\bR$ such that $H(0)=0$ and $H$ is concave
the following identity holds for any $n$:
\begin{eqnarray}  \nonumber
\sum_{1\leq i\leq j\leq n}(\theta_i+\cdots +\theta_j) \left[
H(t_j-t_{i-1})-H(t_j-t_i)-H(t_{j+1}-t_{i-1})+H(t_{j+1}-t_i)\right] && \\  \label{identity}
& & \hspace{-1.5in} = \; \theta_1+\cdots+\theta_n.
\end{eqnarray}
Therefore, if $\psi (\theta)$ contains a drift term, i.e.\ if
$\psi(\theta) = i\beta \theta + \psi_0(\theta)$ where
\begin{equation*}
\psi_0(\theta) := -\frac{1}{2}\sigma^2 \theta^2 + \int_{\mathbb{R}\setminus\{0\}}
\left(e^{i\theta x}-1-i\theta x \ind(|x|<1)\,\right)\nu(dx)
\end{equation*}
(\ref{main}) becomes
\begin{eqnarray*}
\log \phi_n(\theta_1,\ldots,\theta_n) &=& \beta (\theta_1+\cdots+\theta_n) \\
&& \hspace{-1.3in} +\sum_{1\leq i\leq j\leq n}
\psi_0(\theta_i + \cdots + \theta_j)
\left[ H(t_j-t_{i-1})-H(t_j-t_i)-H(t_{j+1}-t_{i-1})+H(t_{j+1}-t_i) \right].
\end{eqnarray*}
\end{lemma}
\begin{proof}
The left hand side of (\ref{identity}) can be written as
\begin{equation}  \label{identity1}
\sum_{r=1}^n \theta_r \sum_{1\leq i \leq r \leq j \leq n}
\left[ H(t_j-t_{i-1})-H(t_j-t_i)-H(t_{j+1}-t_{i-1})+H(t_{j+1}-t_i)\right]
\end{equation}
and the inner sum in (\ref{identity1}) can be written as
\begin{eqnarray} \label{identity2}
 \sum_{1\leq i \leq r}  \left\{ \sum_{r \leq j \leq n} \left[  H(t_j-t_{i-1})-H(t_{j+1}-t_{i-1})  \right] +
\sum_{r \leq j \leq n} \left[ H(t_{j+1}-t_i) - H(t_j-t_i) \right] \right\}  . &&
\end{eqnarray}
The two inner sums in (\ref{identity2}) are telescopic and reduce to
\[
H(t_r-t_{i-1}) - H(t_{n+1}-t_{i-1}) \;=\; H(t_r-t_{i-1}) - H(\infty) \;=\; H(t_r-t_{i-1}) - 1
\]
and
\[
H(t_{n+1}-t_i ) - H(t_r-t_i )   \;=\; H(\infty) - H(t_r-t_i )  \;=\; 1-  H(t_r-t_i )
\]
respectively. (In the above we have taken into account the convention $t_{n+1} = + \infty$.)
Thus (\ref{identity2}) becomes
\begin{eqnarray} \nonumber
\sum_{1\leq i \leq r}  \left\{ H(t_r-t_{i-1}) -  H(t_r-t_i ) \right\}  &=& H(t_r - t_0) - H(t_r-t_r) \\
&& \hspace{-0.8in} \;=\; H(t_r - (-\infty)) - H(0)     \label{identity3}
\;=\;  H(\infty)-0 \;=\;1
\end{eqnarray}
and hence (\ref{identity1}) reduces to $\sum_{r=1}^n \theta_r$  and the proof is complete.
\end{proof}

\begin{proposition}  \label{prop:superposition}
If $\{X_t;t\geq 0\}$, $\{Y_t; t \geq 0\}$, are {\em independent} GCID processes with
characteristic exponents
$\psi(\theta)$ and $c \psi(\theta)$ respectively (where $c>0$) and correlation structure
functions $H_1$ and $H_2$ respectively, then $\{Z_t; t\geq 0\}$ where $Z_t=X_t+Y_t$
is also GCID with characteristic exponent $(1+c)\psi(\theta)$ and correlation structure
function $H(t):=\frac{1}{1+c}H_1(t)+ \frac{c}{1+c} H_2(t)$.
\end{proposition}
\begin{proof}
From Theorem \ref{basic-theorem-H} and the independence of the processes $\{X_t\}$,
$\{Y_t\}$, we see that $\log \bE [\exp(\theta_1Z_{t_1}+\cdots+\theta_nZ_{t_n})]$ is a sum of
terms of the form
\begin{eqnarray*}
&& \psi(\theta_i+\cdots+\theta_j)\big[H_1(t_j-t_{i-1})- \cdots\big] +c\psi(\theta_i+\cdots+\theta_j)
\big[H_2(t_j-t_{i-1})- \cdots\big] \\
&& \hspace{0.5in} \;=\; (1+c)\psi(\theta_i+\cdots+\theta_j)
\left[\frac{1}{1+c}H_1(t_j-t_{i-1}) +\frac{c}{1+c}H_2(t_j-t_{i-1})- \cdots\right] .
\end{eqnarray*}
Clearly $(1+c)\psi$ is a characteristic exponent of an infinitely divisible distribution
when $\psi$ is. Also $H:=\frac{1}{1+c}H_1+ \frac{c}{1+c} H_2$ is a concave
distribution function on $[0,\infty)$ with $H(0)=0$ whenever $H_1$ and $H_2$ are.
\end{proof}

We end the section by defining a subclass of GCID processes
with spectrally positive L\'evy marginal distribution and correlation structure
similar to that of the classical Ornstein-Uhlenbeck process.
\begin{definition} \label{def:ESPC}
A process $\left\{X_t ; t\in \bR\right\}$ will be called an {\em Exponential Spectrally
Positive Coverage Processes} (ESPC) if
\vspace{-0.1in}
\begin{itemize}
\item[(a)] For all $t\in \mathbb{R}$ $X_t $ has spectrally positive L\'{e}vy
distribution with
\begin{equation*}
\psi(\theta) \;:=\; \log \bE  \left[ e^{i\theta X_t}\right]
\;=\; \int_0^\infty (e^{i\theta x}-1)\nu(dx)
\end{equation*}
where $\int_0^\infty x\nu(dx)<\infty $ (but $\int_0^\infty \nu(dx)$
may or may not be finite).
\item[(b)] The finite dimensional distributions of $\{X_t\}$, for $t_1<\cdots<t_n$, are given by
\begin{eqnarray}   \label{SpPos}
\hspace{-0.3in} \log \bE  \left[ e^{i(\theta_1 X_{t_1}
+ \theta_2 X_{t_2} + \cdots +\theta_n X_{t_n})}\right] && \\ \nonumber
&&\hspace{-2.25in} = \hspace{-0.1in}
\sum_{1\leq j\leq k\leq n} \int_0^\infty
\left( e^{i(\theta_j +\cdots + \theta_k)x}-1\right) \nu(dx) \,
(1-e^{-\mu (t_j-t_{j-1})})e^{-\mu (t_k-t_j)}(1-e^{-\mu (t_{k+1}-t_k)}) .
\end{eqnarray}
\end{itemize}
The class of ESPC processes is a subclass of the GCID processes defined in
Theorem \ref{basic-theorem-H}.
\end{definition}
The fact that (\ref{SpPos}) indeed defines a consistent family of finitely dimensional
distributions is a direct consequence of Theorem \ref{basic-theorem-H}.
Note that the mean value of the above process is given by
$\bE  X_t=\int_0^\infty x\nu(dx)$, its variance by
$\mathsf{Var}(X_t)=\int_0^\infty x^2 \nu(dx)$ (provided this last integral is finite),
and its covariance by $\mathsf{Cov}(X_{t_1},X_{t_2})
= e^{-\mu (t_2-t_1)}\int_0^\infty x^2 \nu(dx)$ under the same proviso. The connection
of these processes to limits of superpositions of ON/OFF processes will be discussed
in section \ref{sec:ONOFFlim}.


\section[Poisson marginals and general correlation structure]{Processes with Poisson
and compound Poisson marginals. The $M/GI/\infty$ process and its
generalizations}

In this section we consider processes that fall into the framework of Theorem
\ref{basic-theorem-H} arising from infinite server queues with Poisson arrivals and
general service times. In the first subsection we consider the number of customers
in the system and thus obtain a stationary process with Poisson marginals,
while in the second we will introduce ``marked'' $M/GI/\infty$ queues
where to each customer a real--valued mark is associated. In this fashion
we obtain stationary processes with compound Poisson marginal distributions.

\subsection{The number of customers in the system} \label{sec:mg1}

Suppose that $\left\{ T_l ;l\in \mathbb{Z}\right\} $ is
a Poisson process on the real line with intensity $\lambda$.  $T_l$ denotes the
arrival epoch of the $l$th customer and $\sigma_l$ the customer's service
requirement. We assume that
$\{\sigma_l; l \in \mathbb{Z}\}$ is an i.i.d. sequence of positive
random variables with distribution function $G(x) := \bP(\sigma_l \leq x)$
with finite mean $m=\int_0^\infty (1-G(x))dx$ and independent of the Poisson
process $\{T_l\}$. We will use $\overline{G}:=1-G$ to designate the
corresponding survival function.
The number of customers in the $M/GI/\infty$ system at time $t$ is then
\begin{equation*}
X_t \;=\; \sum_{l\in \mathbb{Z}}
{\bf 1}\left(T_l \leq t < T_l +\sigma_l \right) , \hspace{0.1in} t \in \mathbb{R}.
\end{equation*}
A particularly useful point of view is to consider a Poisson process $\mathcal{M}$
on the upper half plane $\bR\times \bR^+$ with points $\{M_l; l\in\bZ\}$  defined
by $M_l:=(T_l,\sigma_l)$ and mean measure
$\mu(dt\times dx):=\lambda dt \times G(dx)$.
The number of customers in the system at time $t$ is the number of points of
$\mathcal{M}$ in the wedge $A:=\{(s,x)\in \bR \times \bR^+ : s< t<s+x\}$ and this
is a Poisson random variable with mean $\mu(A) = \iint_{ s< t<s+x}
\lambda ds G(dx) = \int_{s=-\infty}^t \lambda ds \int_{x=t-s}^\infty G(dx) =
\lambda \int_{-\infty}^t \overline{G}(t-s) ds = \lambda m$.

Let $-\infty =t_0 < t_1 < t_2 <\cdots < t_n < t_{n+1} = +\infty$ and denote by
\begin{equation*}
\phi_n(\theta_1,\ldots ,\theta_n;t_1,\ldots ,t_n)
:=\bE  [e^{i\theta_1 X_{t_1}+ i\theta_2X_{t_2}+\cdots +i\theta_nX_{t_n}}]
\end{equation*}
the joint characteristic function of $(X_{t_1},\ldots,X_{t_n})$. For
$1\leq i\leq j\leq n,$ denote by $Y_{ij}$ the number of customers who arrived
in the system in the time interval $(t_{i-1},t_i]$ and left in the time
interval $(t_j,t_{j+1}]$. Also, let
$A_{ij}:=\{(x,t):t_j<x+t\leq t_{j+1};x\geq 0;t_{i-1}<t\leq t_i\}$ and
\begin{equation}   \label{muAij}
\mu (A_{ij}) = \iint_{A_{ij}}\lambda dt G(dx).
\end{equation}
Then the collection of random variables $\{Y_{ij};1\leq i\leq j\leq n\}$ are
independent, Poisson distributed, with parameters $\mu (A_{ij})$.
Furthermore,
\begin{equation*}
X_{t_k}=\sum_{j=k}^n \sum_{i=1}^k Y_{ij}
= \sum_{j=1}^n \sum_{i=1}^n Y_{ij} \ind(i \leq k \leq j)
\end{equation*}
and
\[
\sum_{k=1}^n \theta_k X_{t_k} =
\sum_{k=1}^n \theta_k \sum_{j=1}^n \sum_{i=1}^n Y_{ij} \ind(i \leq k \leq j) =
\sum_{j=1}^n \sum_{i=1}^n Y_{ij} \sum_{k=1}^n \theta_k \ind(i \leq k \leq j) =
\hspace{-0.05in}
\sum_{1\leq i \leq j \leq n} \hspace{-0.05in} \left(\theta_i +\cdots + \theta_j\right) Y_{ij}
\]
hence
\begin{eqnarray}
\phi_n(\theta_1,\ldots ,\theta_n; t_1,\ldots ,t_n) &=&
\bE  \left[ \prod_{1\leq i\leq j\leq n}e^{i(\theta_i+\cdots +\theta_j)Y_{ij}} \right] \;=\;
\prod_{1\leq i\leq j\leq n}\bE  \left[ e^{i(\theta_i +\cdots +\theta_j)Y_{ij}}\right]
\notag \\
&=&\prod_{1\leq i\leq j\leq n}e^{\mu (A_{ij})
\left( e^{i(\theta_i+\cdots +\theta_j)}-1\right)}  \notag \\
&=&\exp\left( \sum_{1\leq i\leq j\leq n}\mu (A_{ij})
\left(e^{i(\theta_i+\cdots +\theta_j)}-1\right) \right) .
\label{ch.f}
\end{eqnarray}
From (\ref{muAij})
\begin{eqnarray*}
\mu (A_{ij}) &=&\lambda \int_{t=t_{i-1}}^{t_i}
\left[\int_{x=t_j-t}^{t_{j+1}-t}dG(x)\right] dt \\
&=&\lambda m
\left[ \frac{1}{m}\int_{t=t_{i-1}}^{t_i}\overline{G}(t_j-t)dt
-\frac{1}{m}\int_{t=t_{i-1}}^{t_i}\overline{G}(t_{j+1}-t)dt \right]
\end{eqnarray*}
or
\begin{equation}  \label{mij}
\mu (A_{ij})=\rho \left[ G_I(t_j-t_{i-1})-G_I(t_j-t_i)-G_I(t_{j+1}-t_{i-1})
+G_I(t_{j+1}-t_i) \right]
\end{equation}
where, in the above expression, $\rho =\lambda m$ and we have used the
integrated tail distribution given by
\begin{equation}   \label{GIx}
G_I(x):=\frac{1}{m}\int_0^x \overline{G}(y)dy.
\end{equation}
Therefore this is also a GCID process with $H(t)=G_I(t)$ and
$\psi (\theta )=\rho \left(e^{i\theta }-1\right)$.
The supermodularity of $G_I(t)$ follows from its concavity which, in turn, is
a consequence of the fact that it is an integral of an non--increasing
function. The marginal distribution of $X_t$ is of course Poisson with
mean $\rho$.

Of particular importance is the case where $G(t)=1-e^{-\mu t}$. In this case $G_I(t)=G(t)$
for all $t$ and the right hand side of (\ref{mij}) can be writen as
$\left(1-e^{-\mu (t_i-t_{i-1})}\right)e^{-\mu (t_j-t_i)}\left(1-e^{-\mu(t_{j+1}-t_j)}\right)$.
The resulting process
is of the ESPC type with joint characteristic function given by
\begin{eqnarray}   \label{MMinfty}
\log \phi_n(\theta_1,\ldots,\theta_n;t_1,\ldots,t_n) && \\   \nonumber
&& \hspace{-1.4in} =\sum_{1\leq i\leq j\leq n}\rho
\left( e^{i(\theta_i +\cdots +\theta_j)}-1\right)
\left(1-e^{-\mu (t_i-t_{i-1})}\right)e^{-\mu (t_j-t_i)}\left(1-e^{-\mu(t_{j+1}-t_j)}\right).
\end{eqnarray}
This corresponds to the $M/M/\infty$ case and is a special case of and ESPC process
(\ref{SpPos}) with $\psi (\theta )=\rho \left( e^{i\theta }-1\right)$.

\subsection{The marked $M/GI/\infty $ system}

The process in section \ref{sec:mg1} can be generalized by considering a marked
$M/GI/\infty $ queue. Suppose that $\left\{ T_l ; l\in \mathbb{Z}\right\} $ is
a Poisson process with intensity $\lambda$ defined on the real line and
$\{\sigma_l; l \in \mathbb{Z}\}$ a sequence of independent, identically distributed
positive random variables with distribution function $G(x) := \bP(\sigma_l \leq x)$
with finite mean $m$. $\sigma_l$ is to be thought of as the service
requirement of the $l$th customer who arrives in the system (at time $T_l$).
Also, let $\{\eta_l; l \in \mathbb{Z}\}$ be an i.i.d.\ sequence of random variables
representing the mark process. $\eta_l$ is a real-valued mark corresponding to some
attribute of the $l$th customer and the characteristic function of the marks is denoted by
\begin{equation}    \label{mrgnl}
	\chi (\theta) := \bE   e^{i\theta \eta }.
\end{equation}
Stationary ``marked $M/GI/\infty$ process''
\begin{equation*}
X_t \;=\; \sum_{n\in \mathbb{Z}}
{\bf 1}\left(T_l \leq t < T_l +\sigma_l \right) \eta_l, \hspace{0.1in} t \in \mathbb{R}.
\end{equation*}
In view of the above assumptions its marginal distribution is
\begin{equation*}
\bE  \left[ e^{i\theta X_t}\right] = e^{\rho \big( \chi(\theta)-1\big)}
\end{equation*}
where $\rho = \lambda m$. (This is of course a compound Poisson distribution.)
The joint characteristic function for the finite dimensional distributions of this process
is given by the following generalization of (\ref{ch.f})
\begin{equation}  \label{marked-MG}
\phi_n (\theta_1 ,\ldots ,\theta_n ;t_1 ,\ldots,t_n)
= e^{-\sum_{1\leq i\leq j\leq n}\mu (A_{ij})\left( 1-\chi (\theta_i +\cdots +\theta_j)\right)}
\end{equation}
where the $\mu (A_{ij})$ are given by (\ref{mij}). Thus, (\ref{marked-MG}) corresponds to
(\ref{main}) with $H(t) = G_I(t)$ given by (\ref{GIx}) and $\psi(\theta) = \rho\big(\chi(\theta)-1\big)$
where $\chi(\theta)$ is given by (\ref{mrgnl}).

\section{Triangular arrays of ON/OFF sources and limiting theorems} \label{sec:ONOFFlim}
In this section we consider triangular arrays of ON/OFF sources and obtain conditions under
which their superposition converges in distribution to an ESPC process (Definition \ref{def:ESPC}).
There is a large literature on the use of ON/OFF sources in developing models of
internet traffic that provide an explanation for the statistical characteristics observed
and in particular for the Long Range Dependence, self similarity, and heavy tailed
behavior which characterizes it. Limit theorems for the superposition of ON/OFF
sources and their corresponding cumulative input processes have been studied in
great detail in the 1990's and  2000's.
Two types of models, particularly popular for modeling high speed data traffic
and explaining its statistical characteristics,
are the superposition of $M$ ON/OFF processes and the infinite source
Poisson model (also referred to as the $M/G/\infty$ input model). Heath, Resnick and
Samorodnitsky \cite{HeathResnickSamorodnitsky} and Mikosch et al. \cite{mikosch}
show that, depending on the scaling behavior one may obtain for the cumulative input
process either fractional Brownian motion (fBm) with Hurst coefficient in $(\frac{1}{2},1)$
i.e.\ a Gaussian process with Long Range Dependence or stable L\'evy motion.
We also mention and
Mikosch and Samorodnitsky \cite{MKSA} which extend the results for ON/OFF processes
described by more general marked point processes.
For convergence to L\'evy motion see also Konstantopoulos and Lin \cite{Takis},
Maulik  and Resnick \cite{MaulikResnick}, and
Resnick and van den Berg \cite{ResnickBerg}.
Resnick and Samorodnitsky \cite{ResnickSamorodnitsky} consider hierarchical models and
Pipiras and Taqqu \cite{Pipiras} examine models based on renewal-reward processes.
More general models of this type are fall within the framework of the {\em Telecom process}
(see  Wolpert and Taqqu \cite{WolpertTaqqu}, Kaj and Taqqu \cite{KajTaqqu}).



Consider a continuous time Markov process $\left\{ \zeta_t ;t\geq 0\right\} $
with state space $\{0,r\}$ and generator
\[  
Q=\left[ \begin{array}{cc} -\lambda & \lambda \\ \mu & -\mu \end{array} \right] .
\]  
The transition probability matrix of this process is
\begin{equation}   \label{TransProb}
P(t) \;=\; \left[
\begin{array}{ll}
\frac{\mu }{\lambda +\mu } & \frac{\lambda }{\lambda +\mu } \\
\frac{\mu }{\lambda +\mu } & \frac{\lambda }{\lambda +\mu }
\end{array}
\right] +e^{-(\lambda+\mu) t}\left[
\begin{array}{rr}
\frac{\lambda }{\lambda +\mu } & -\frac{\lambda }{\lambda +\mu } \\
-\frac{\mu }{\lambda +\mu } & \frac{\mu }{\lambda +\mu }
\end{array}
\right] .
\end{equation}
If $\bP(\zeta_0 = 0)=\frac{\mu}{\lambda+\mu}$, $\bP(\zeta_0 = r)
 = \frac{\lambda}{\lambda+\mu}$, the process $\{\zeta_t\}$ is
stationary. We will refer to $\lambda$ as the {\em OFF rate}, $\mu$ as the
{\em ON rate} and $r>0$ as the {\em intensity} of the source. The triplet
$(\lambda,\mu,r)$ characterizes completely this ON/OFF source. In this section
we give results that show how the class of processes defined by (\ref{main})
can arise as a limit of superpositions of ON/OFF processes of this type with
varying parameters.

Consider next a collection of $n$ sources each of which generates intermittent
traffic: Suppose that the $j$th source is inactive during an exponential period with
rate $\lambda_j$ and then becomes active for an
independent period of time, exponentially distributed, with rate $\mu$,
the same for all sources. When the $j$th source is ON, it generates fluid at
rate $r_j$, whereas when it is OFF it does not generate any fluid.
We consider the effect of the superposition of a large number of such sources
under the assumption that the sources are independent and that each individual
source generates traffic infrequently.

In order to study the asymptotic behavior of such a superposition let us introduce
a triangular array of independent, markovian ON/OFF processes,
$\left\{ \zeta_{nj}(t);t\geq 0\right\}$, $j=1\ldots n$, $n\in \mathbb{N}$, with
corresponding triplets $(\lambda_{nj},\mu,r_{nj})$.
Note that we have made the simplifying assumption that the ON rates are the same
for all elements of the array. For simplicity we will assume the stationarity of the
elements of the family
$\{\zeta_{nj}(t)\}$ under the probability measure $\bP$. Denote by
\begin{equation} \label{Xnt1}
	X_n(t) \;=\; \sum_{j=1}^n \zeta_{nj}(t), \;\;\;\; t \in [0,T], \;\; n=1,2,\ldots
\end{equation}
the sequence of superpositions of the ON/OFF processes. These have sample
paths in the space $\mathbb{D}[0,T]$ consisting of the \textit{c\`adl\`ag}
functions on $[0,T]$, i.e.\ all functions that are continuous from the right
and have left limits.
At a fixed point in time, $t$, we obtain a double array of independent random variables
\begin{equation}    \label{zetanj}
\zeta_{nj}(t) \;=\; \left\{ \begin{array}{ccc} 0 & \mbox{ with prob. } &
\frac{\mu}{\lambda_{nj} + \mu}, \\   & & \\
r_{nj} & \mbox{ with prob. } &
\frac{\lambda_{nj}}{\lambda_{nj} + \mu}  \end{array} \right.
\end{equation}
with corresponding characteristic functions
$ 
\bE e^{i\theta \zeta_{nj}(t)} \;=\; \frac{\mu}{\lambda_{nj} + \mu}
+ \frac{\lambda_{nj}}{\lambda_{nj} + \mu} \, e^{i \theta r_{nj}}.
$ 

A double array of independent random variables $\{\chi_{nj}\}$, $j=1,\ldots,n$,
$n \in \bN$, of real random variables is called {\em uniformly asymptotically negligible}
or {\em u.a.n.} (alternatively {\em null}, see \cite{Kallenberg}) if
\begin{equation}   \label{uan}
\lim_{\epsilon \rightarrow 0} \max_{1\leq j \leq n} \bP(|\chi_{nj}| > \epsilon) \;=\; 0.
\end{equation}
If we denote by $F_{nj}(x):= \bP(\chi_{nj}\leq x)$, $x \in \bR$, the distribution functions
of the elements of the triagular array and by
$\phi_{nj}(\theta):= \int_\bR e^{i\theta x} F_{nj}(dx)$, $\theta \in \bR$, the corresponding
characteristic functions, a necessary and sufficient condition for the triangular array to be
{\em u.a.n.} is (see \cite[p.305]{FristedtGrey}) 
\begin{equation}  \label{uan1}
\lim_{n\rightarrow \infty} \max_{1\leq j \leq n} \left| \phi_{nj}(\theta) -1 \right| = 0
\hspace{0.2in} \mbox{ for all $\theta \in \bR$}.
\end{equation}
The following result (Theorem 24 in \cite[p.311]{FristedtGrey}) gives necessary and sufficient
conditions for the convergence in distribution of the row sums of such a double array.
\begin{theorem} \label{th:FrG} Let $\{\chi_{nj}\}$, $j=1,\ldots,n$, $n\in \bN$, be
a u.a.n. array of nonnegative real random variables. In order that the sequence
of row sums, $W_n:=  \sum_{j=1}^n \chi_{nj}$, converge in distribution it is necessary and
sufficient that there exist $\gamma >0$ and a L\'evy measure $\nu$ on $\bR^+$ satisfying
the conditions
\begin{itemize}
\item[C1.] $\displaystyle \nu[x,\infty) = \lim_{n\rightarrow \infty}
\sum_{j=1}^n \bP(\chi_{nj}\geq x)$ if $x >0$ and $\nu\{x\}=0$,
\item[C2.]
$\displaystyle
\gamma \;=\; \lim_{\epsilon \rightarrow 0} \limsup_{n\rightarrow \infty}
\sum_{j=1}^n \bE[\chi_{nj}\ind(\chi_{nj} \leq \epsilon)] \;=\;
\lim_{\epsilon \rightarrow 0} \liminf_{n\rightarrow \infty}
\sum_{j=1}^n \bE[\chi_{nj}\ind(\chi_{nj} \leq \epsilon)]  .
$
\end{itemize}
When these conditions are satisfied, the sequence $\{W_{n}\}$ converges
in distribution to an infinitely divisible random variable with characteristic
function $\exp\left(\gamma \theta+ \int_0^\infty (e^{i\theta x} - 1) \nu(dx)\right)$.
\end{theorem}

We will discuss sufficient conditions for the array $\{\zeta_{nj}(t)\}$ to be u.a.n. (for
each value of $t$) and for the sequence of processes defined in (\ref{Xnt1}) to converge
in distribution to a limiting process.
Since we assume for the sake of simplicity that all sources have the same ON rate, $\mu$,
the triangular array of input processes $\{\zeta_{nj}\}$ is described by a corresponding
triangular array of parameters $(\lambda_{nj}, r_{nj})$, $j=1\ldots n$, $n\in \mathbb{N}$.
Let $\delta_x$ denote the Dirac measure on the real line placing a unit mass at
$x \in \mathbb{R}$, defined by $\delta_{x}(A)=1$ if $x\in A$ and $0$ otherwise
for every Borel $A\subset \mathbb{R}$. We will describe the limiting behavior
of the input process (\ref{Xnt1}) as $n\rightarrow \infty$ by considering the
corresponding behavior of the sequence of measures
\begin{equation}  \label{seqmeas2}
	\nu_n (\cdot ) = \sum_{j=1}^n \lambda_{nj}\delta_{r_{nj}}(\cdot).
\end{equation}

We will show that, when the sequence of measures
$\{\nu_n\}$ converges weakly to a measure
$\nu$ on $(0,\infty)$, under
some additional assumptions, the sequence of processes $\{X_n\}$ defined
in (\ref{Xnt1}) converges
weakly to a limit process $X$ of the GCID type defined in Theorem
\ref{basic-theorem-H}.

Suppose that $\nu$ is a $\sigma-$finite measure on $(0,\infty)$ which satisfies the condition
$\int_0^\infty (x \wedge 1) \nu(dx) < \infty$ and thus is a L\'evy measure on $(0,\infty)$.
In particular $\nu[x,\infty) \,<\,\infty$ for each $x>0$.

We assume that the array $\{(\lambda_{nj},r_{nj})\}$ satisfies the conditions
\begin{eqnarray*}
(A.1) && \;\underset{1\leq j\leq n}{\max } \lambda_{nj}
\rightarrow 0 \;\; \mbox{ when $n\rightarrow \infty$}, \\
(A.2) &&\; \lim_{\epsilon \downarrow 0} \limsup_{n\rightarrow \infty} \sum_{j=1}^n \lambda_{nj}
r_{nj}\ind(r_{nj} \leq \epsilon) \;= \;0, \\
(A.3) && \;
       \sup_{n \in \mathbb{N}} \sum_{j=1}^n \lambda_{nj}r_{nj}^p \leq C ,
       \hspace{0.1in} \mbox{ for $p\in [1,4]$ and for some $C>0$}.
\end{eqnarray*}
In addition assume that the double array of parameters $\{(\lambda_{nj},r_{nj})\}$,
$j=1,\ldots,n$, $n\in \bN$, is such that the sequence of empirical measures $\{\nu_n\}$
converges weakly to $\nu$, i.e.
\[
(A.4) \hspace{0.08in} \nu_n[x,\infty) \,= \,\sum_{j=1}^n \lambda_{nj}\ind(r_{nj} \geq x)
\; \rightarrow \; \nu[x,\infty) \; \mbox{ as $n \rightarrow \infty$,} \;\mbox{ for every $x>0$
for which $\nu\{x\} =0$}.
\]
The next assumption concerns the measure $\nu$:
\begin{eqnarray*}
\hspace{-1.2in} (A.5) && \int_0^\infty x^q \nu(dx) \; < \; \infty \;\;\;\; q=1,2.
\end{eqnarray*}
Finally, we will assume that
\begin{eqnarray*}
\hspace{-1in} (A.6) && \hspace{0.2in} \int_{[x,\infty)}y \nu_n(dy)
\,= \,\sum_{j=1}^n \lambda_{nj}r_{nj}\ind(r_{nj} \geq x)
\; \rightarrow \; \int_{[x,\infty)}y\nu(dy) \;\; \mbox{ as $n \rightarrow \infty$,} \hspace{0.2in} \\
&& \hspace{3in}  \;\mbox{ for every $x >0$ for which $\nu\{x\} =0$}.
\end{eqnarray*}
\begin{proposition} \label{prop:empirical_L}
Assumptions (A.4) and (A.6) taken together imply the following
\[
(A.6b) \hspace{0.08in} \int_0^\infty \left( e^{i\theta y} -1 \right) \nu_n(dy) \,= \,
\sum_{j=1}^n \lambda_{nj} \left( e^{i\theta r_{nj}} -1 \right)
\; \rightarrow \; \int_0^\infty \left( e^{i\theta y} -1 \right) \nu(dy) \; \mbox{ as $n \rightarrow \infty$,}
\mbox{  $ \forall \theta \in \bR$}.
\]
\end{proposition}
\begin{proof}
Define the sequence of measures $\{\widetilde{\nu}_n\}$, $1,2,\ldots$ on $(0,\infty)$,
via $\widetilde{\nu}_n (dy) := y\nu_n(dy)$. These measures are {\em finite} and the
sequence $\{\widetilde{\nu}_n\}$ converges weakly to the measure $\widetilde{\nu}$ 
defined by $\widetilde{\nu}(dy):=y \nu(dy)$ as a result of (A.6). Further, $\widetilde{\nu}$ 
is finite as a result of (A.5). Define the family of functions $h_\theta:\bR \rightarrow \bC$,
indexed by $\theta\in \bR$,
\[
h_\theta(x) := \left\{ \begin{array}{cl} \frac{e^{i\theta x} -1}{x} & \mbox{ if $x \ne 0$} \\
i \theta  & \mbox{ if $x=0$} \end{array}\right. \;.
\]
It can be readily seen that, for each $\theta$, $h_\theta$ is continuous and bounded and
hence from Helly's second theorem
\[
\lim_{n\rightarrow \infty} \int_0^\infty h_\theta(y) \widetilde{\nu}_n(dy)  \;=\;
\int_0^\infty h_\theta(y) \widetilde{\nu}(dy).
\]
This is however equivalent to the convergence in (A.6b).
\end{proof}


An example of a double array of parameters and corresponding ON/OFF processes satisfying
Asssumptions A.1-A.6 and Conditions C.1 and C.2 is given in the Appendix.

The following proposition gives necessary and sufficient conditions for the convergence of
the row sums of the triangular array $\{\zeta_{nj}(t)\}$.
\begin{proposition} \label{prop:marginal}
Suppose that there exists a L\'evy measure $\nu$ on $(0,\infty)$ for which
the triangular array of parameters $(\lambda_{nj},r_{nj})$ satisfies Assumptions
A.1, A.2, and A.4. Then for each fixed $t$
\begin{equation}
X_n(t):=\sum_{j=1}^n \zeta_{nj}(t) \stackrel{d}{\rightarrow} X_t, \hspace{0.2in}
\mbox{ as $n\rightarrow \infty$.}
\end{equation}
$X_t$ is infinitely divisible with
\[
\log  \mathbb{E}[e^{i\theta X_t}] \;=\; \int_0^\infty \big(e^{i\theta x} - 1 \big) \frac{1}{\mu} \nu(dx).
\]
\end{proposition}
\begin{proof}
Since
\begin{equation} \label{on_off_null}
\max_{1\leq j \leq n} \left| \bE e^{i\theta \zeta_{nj}(t)} - 1\right| \;=\;
\max_{1\leq j \leq n} \frac{\lambda_{nj}}{\lambda_{nj} + \mu}
\left| e^{i\theta r_{nj}} -1 \right| \;\leq \; \frac{2}{\mu} \,
\left( \max_{1 \leq j \leq n} \lambda_{nj} \right) \rightarrow 0 \;\;
\mbox{ as $n\rightarrow \infty$}
\end{equation}
from Assumption A.1 and (\ref{uan1}) the array $\{\zeta_{nj}(t)\}$ is null for each fixed $t$.
For each $x >0$
\begin{equation}  \label{A0A2}
\sum_{j=1}^n \bP(\zeta_{nj}(t) \geq x) = \sum_{j=1}^n
\frac{\lambda_{nj}}{\lambda_{nj}+\mu} \ind(r_{nj} \geq x) =
\frac{1}{\mu} \sum_{j=1}^n \lambda_{nj}  \ind(r_{nj} \geq x)
-  \sum_{j=1}^n \frac{\lambda_{nj}^2}{\mu(\lambda_{nj}+\mu)}  \ind(r_{nj} \geq x) .
\end{equation}
The second term in the right hand side of the last equation vanishes as $n \rightarrow \infty$ because
\[
0 \leq \sum_{j=1}^n \frac{\lambda_{nj}^2}{\mu(\lambda_{nj}+\mu)}  \ind(r_{nj} \geq x)
\leq \left( \max_{1\leq j \leq n} \lambda_{nj} \right) \frac{1}{\mu^2} \sum_{j=1}^n \lambda_{nj} \ind(r_{nj} \geq x)
\rightarrow 0 \;\;\; \mbox{for all $x >0$}
\]
by Assumption A.1 and A.4.
Hence, again by Assumption A.4,
\[
\lim_{n\rightarrow \infty} \sum_{j=1}^n \bP(\zeta_{nj}(t) \geq x) =  \frac{1}{\mu} \nu[x,\infty)
\]
where $\frac{1}{\mu} \nu(\cdot)$ is a L\'evy measure on $(0,\infty)$. Also,
\[
\sum_{j=1}^n \bE[\zeta_{nj}(t) \ind(\zeta_{nj} (t) \leq \epsilon]
 \;=\;  \sum_{j=1}^n \frac{\lambda_{nj}}{\lambda_{nj}+\mu} r_{nj} \ind(r_{nj} \leq \epsilon)
 \; \leq \;  \frac{1}{\mu} \sum_{j=1}^n \lambda_{nj} r_{nj}  \ind(r_{nj} \leq \epsilon)
\]
and hence A.2 implies C.2 with $\gamma=0$.
Hence, as an immediate consequence of Theorem \ref{th:FrG},
\begin{equation} \label{chflimit}
\lim_{n\rightarrow \infty} \log \bE [e^{i \theta X_n(t) }]
= \frac{1}{\mu} \int_0^\infty \left(e^{i\theta x}-1 \right) \nu(dx)
\hspace{0.2in} \mbox{ for any $\theta \in \bR$}.
\end{equation}
\end{proof}

The following theorem in Billingsley \cite[p.142]{Billingsley_M} guarantees the weak
convergence of a family of processes in $\mathbb{D}[0,T]$.
\begin{theorem} \label{th:Billingsley}
Let $\{X_n(t); t \in[0,T]\}$ be a sequence of real-valued processes  and a process
$\{X(t); t\in[0,T]\}$ with sample paths in $\mathbb{D}[0,T]$, such that, for all
finite dimensional distributions
\begin{equation} \label{Billingsley-fidi}
\left(X_n(t_1),\ldots,X_n(t_k) \right) \stackrel{d}{\rightarrow} \left(X(t_1),\ldots,X(t_k)\right) \;\;\;
\mbox{for all $k\in \mathbb{N}$ and $t_i \in [0,T]$}
\end{equation}
 and further that
\begin{equation}  \label{tight-T}
X(T) - X(T-\delta) \; \stackrel{d}{\rightarrow} 0 \;\; \mbox{ as $\delta \rightarrow 0$}.
\end{equation}
	Finally, suppose that for $t_1 \leq t_2 \leq t_3 $ the inequality
\begin{equation}   \label{Billingsley-ineq}
		\bE\left[ \left\vert X_n (t_2)-X_n (t_1)\right\vert^\beta
		\left\vert X_n (t_3)-X_n (t_2)\right\vert^\beta \right] \leq K (t_3 -t_1)^{1+\alpha}
\end{equation}
hods for some $\alpha >0$, $\beta >1$, and $K>0$. Then $X_n \stackrel{d}{\rightarrow} X$ in
$\mathbb{D}[0,T]$.
\end{theorem}
\begin{theorem}\label{th:array-convergence}
Suppose that there exists a L\'evy measure $\nu$ on $(0,\infty)$ for which
the triangular array of parameters $(\lambda_{nj},r_{nj})$ satisfies Assumptions
A.1 -- A.6. Then the row sums of the triangular array of processes 
$\{\zeta_{nj}(t); t\geq 0\}$ satisfy
\begin{equation}
X_n(t) :=\sum_{i=1}^n \zeta_{nj}(t) \stackrel{d}{\rightarrow} X(t)
\end{equation}
where $\{X(t) ;t\geq 0\}$ is an ESPC process (see Definition \ref{def:ESPC}) with finite
dimensional distributions given by
\begin{eqnarray}  \nonumber
&&\log \bE  \left[ e^{i\left(\theta_1 X(t_1)+\theta_2 X(t_2)+\cdots +\theta_m X(t_m )\right)}\right]
\; =\; \frac{1}{\mu}\sum_{1\leq j\leq k\leq m}\int_0^\infty
\left(e^{i\left( \theta_j+\cdots +\theta_k \right) x}-1\right) \, \nu(dx) \,  \\ \label{GOU}
&& \hspace{2.3in} \times \left(1-e^{-\mu (t_j-t_{j-1})}\right)
e^{-\mu (t_k-t_j)}\left(1-e^{-\mu(t_{k+1}-t_k)}\right).
\end{eqnarray}
In the above expression $t_0=-\infty$ and  $t_{m+1}=+\infty$ according to the standard convention
in this paper.
\end{theorem}
\begin{proof}
The convergence of the marginal distributions has already been proved (under weaker assumptions) 
in Proposition (\ref{prop:marginal}). We will split the proof in two parts.

\paragraph{Part 1. Convergence of the finite dimensional distributions.}
We begin by determining the joint distribution of the process
$\{X_n(t) \}$ at times $t_1 <t_2 <\cdots <t_m$.
Set $\xi_{nj}(t) = \zeta_{nj}(t)/r_{nj}$ (so that the processes $\xi_{nj}(t)$ take
values in the set $\{0,1\}$).
The joint characteristic function is given by
\begin{eqnarray}  \nonumber
\bE  e^{i\sum_{k=1}^m \theta_k X_n (t_k)} &=& \bE
 e^{i \sum_{j=1}^n r_{nj} \sum_{k=1}^m \theta_k  \xi_{nj}(t_k)} \;=\; \bE  \prod_{j=1}^n
e^{i\sum_{k=1}^m r_{nj} \theta_k \xi_{nj}(t_k)} \\   \label{cf1}
&=& \prod_{j=1}^n \bE  e^{i\sum_{k=1}^m r_{nj} \theta_k \xi_{nj}(t_k)},
\end{eqnarray}
the last equation following from the independence of the sources.
Writing for the generic element of the array $\xi_k :=\xi_{nj}(t_k)$, $r$ in place $r_{nj}$, and
taking into account the fact that
$e^{i\theta_k r\xi_k} \;=\;  1+\xi_k(e^{ir\theta_k}-1)$
since $\xi_k$ takes only the values 0 and 1,
\begin{eqnarray} \nonumber
\bE  e^{i\sum_{k=1}^m  \theta_k \xi_{nj}(t_k)} &=& \bE \prod_{k=1}^m e^{i \theta r \xi_k}
\;=\;\bE \prod_{k=1}^m \left( 1+\xi_k (e^{ir\theta_k }-1)\right)  \\
&&\hspace{-1.3in} =\;1+\sum_{1\leq l_1 \leq m} \bE  \left[ \xi_{l_1}\right]
(e^{ir\theta_{l_1}}-1) \;+\;
\sum_{k=2}^m \hspace{-0.5in}\sum_{\hspace{0.6in}  1\leq l_1 < l_2 <\cdots < l_k \leq m}
\hspace{-0.6in} \bE  \left[ \xi_{l_1}\cdots \xi_{l_k}\right]
(e^{ir\theta_{l_1}}-1)\cdots (e^{ir\theta_{l_k}}-1).  \label{cf2}
\end{eqnarray}
Set
\begin{equation}  \label{notation}
\pi_{nj} := \frac{\lambda_{nj}}{\lambda_{nj} + \mu},\hspace{0.2in} \alpha_{nj}:=\lambda_{nj}+\mu.
\end{equation}
Then
\begin{eqnarray}
\bE  \xi_{nj}(t_1) &=& \pi_{nj}   \nonumber \\
 \nonumber
\bE  [\xi_{nj}(t_{l_1}) \cdots \xi_{nj}(t_{l_k})] &=& \pi_{nj} \left[ e^{-\alpha_{nj}(t_{l_2}-t_{l_1})}+\pi_{nj}
\left( 1-e^{-\alpha_{nj}(t_{l_2}-t_{l_1})}\right) \right] \times \cdots  \\
&&\hspace{0.5in}\times \left[ e^{-\alpha_{nj}(t_{l_k}-t_{l_{k-1}})}
+\pi_{nj}\left(1-e^{-\alpha_{nj}(t_{l_k}-t_{l_{k-1}})}\right) \right]    \label{estimate}
\end{eqnarray}
for $k=2,\ldots,m$. Setting
\begin{equation}  \label{fact3}
L_{nj}(k) := \bE  [\xi_{nj}(t_{l_1}) \cdots \xi_{nj}(t_{l_k})]
- \frac{\lambda_{nj}}{\mu}e^{-\mu (t_{l_k}-t_{l_1})}
\end{equation}
we have from Proposition \ref{prop:boundB}  of the Appendix the inequality
$0 \leq L_{nj}(k) \leq \lambda_{nj}^2 M$. Furthermore, it is also shown there that
(\ref{cf2}) can be written as
\begin{eqnarray} \nonumber
\bE  e^{i\sum_{k=1}^m \theta_k r_{nj} \xi_{nj}(t_k)} &=&
1+\frac{\lambda_{nj}}{\mu}  \sum_{1\leq l_1\leq m} (e^{ir\theta_{l_1 }}-1) \\ \nonumber
&& +\frac{\lambda_{nj}}{\mu} \sum_{k=2}^m \hspace{-0.6in}
\sum_{\hspace{0.7in} 1\leq l_1 <l_2 <\cdots <l_k \leq m}
\hspace{-0.55in} e^{-\mu (t_{l_k}-t_{l_1})}(e^{ir\theta_{l_1 }}-1)
\cdots (e^{ir\theta_{l_k}}-1)  \;+\; R_{nj}
 \\    \nonumber 
\end{eqnarray}
where the quantity $R_{nj}$ is defined in (\ref{ccf2}), satisfies (\ref{ineqccf2}), and
will turn out to be negligible.
Next, use the algebraic identity (\ref{Algebraic-Identity}) to rewrite the double
sum on the right hand side of the above equation and obtain
\begin{eqnarray*}
\bE  e^{i\sum_{k=1}^m \theta_k r_{nj} \xi_{nj}(t_k)}
&=& 1 \;+\; R_{nj} \;+\; \frac{\lambda_{nj}}{\mu}\sum_{u=1}^m(e^{ir_{nj}\theta_u}-1) \\
&& \hspace{-1.5in} +\frac{\lambda_{nj}}{\mu}\sum_{1\leq u<v\leq m}
\left( e^{ir_{nj}(\theta_u+\cdots +\theta_v)}-1\right) \left[ 1-e^{-\mu(t_u-t_{u-1})}\right]
e^{-\mu(t_v-t_u)}\left[ 1-e^{-\mu(t_{v+1}-t_v)}\right] \\
&=&  1 \;+\; R_{nj} \;+\; \frac{\lambda_{nj}}{\mu}\sum_{u=1}^m a_{nj}(u)
+  \frac{\lambda_{nj}}{\mu}\sum_{1\leq u<v\leq m} b_{nj}(u,v)
\end{eqnarray*}
with the obvious identifications for the terms $a_{nj}(u) $ and $b_{nj}(u,v)$. Further, set
$c_{nj}(u,v):= \sum_{u=1}^m a_{nj}(u) + \sum_{1\leq u<v\leq m} b_{nj}(u,v)$.
Then, for the process $X_n (t)=\sum_{j=1}^n r_{nj} \xi_{nj}(t)$ we have, in view of the above,
(\ref{cf1}), and the fact that the logarithms below are well defined for $n$ sufficiently large in view
of Assumption A.1 and (\ref{ineqccf2})
\begin{eqnarray*}
\log \bE  e^{i\sum_{k=1}^m \theta_k X_n (t_{k})}&=& \sum_{j=1}^n \log
\left( \bE  e^{i\sum_{k=1}^m \theta_k \xi_{nj}(t_k)}\right) \\
& =& \sum_{j=1}^n \log \Big( 1+ R_{nj}+  \frac{\lambda_{nj}}{\mu}c(u,v) \Big)
\;=\; \sum_{j=1}^n  \Big( \frac{\lambda_{nj}}{\mu}c_{nj}(u,v) + R_{nj} + \Xi_{nj}\Big)
\end{eqnarray*}
where $\Xi_{nj} =O\big( \lambda^2_{nj} r_{nj}^2\big)$.
Here we have used the inequality $| \log(1+z) - z | \leq |z|^2$ for $|z| \leq \frac{1}{2}$ 
whose use is justified because, by the triangle inequality, 
\[
|c_{nj}(u,v)|  \;\leq \; \sum_{u=1}^m |e^{r_{nj}\theta_ui}-1| 
+  \sum_{1\leq u<v\leq m} |e^{r_{nj}(\theta_u+\cdots+\theta_v)}-1| K_0 
\;\leq \; \min(K_1,r_{nj}K_2), 
\] 
(where $K_0$, $K_1$, $K_2$ are appropriate positive constants) and therefore in view 
of Assumption A.1, there exists $n_0$ such that
$
\max\limits_{1\leq j \leq n} \left| \frac{\lambda_{nj}}{\mu}c_{nj}(u,v) \right| < \frac{1}{2}.
$
Hence as $n \rightarrow \infty$
\begin{eqnarray*}
&&\hspace{-0.3in}\left\vert \log \bE  e^{i\sum_{k=1}^m \theta_k X_n (t_{k})}
 - \sum_{u=1}^m \sum_{j=1}^n \frac{\lambda_{nj}}{\mu}(e^{ir_{nj}\theta_{u}}-1)  \right. \\
&&\left. \hspace{-0.2in} - \hspace{-0.1in}
\sum_{1\leq u <v \leq m}\sum_{j=1}^n \frac{\lambda_{nj}}{\mu}\left( e^{ir_{nj}(\theta_u +\cdots +\theta_v)}-1\right)
\left[ 1-e^{-\mu(t_{u}-t_{u-1})}\right] e^{-\mu(t_v-t_u)}
\left[ 1-e^{-\mu(t_{v+1}-t_{v})}\right] \right\vert \\
&& \hspace{-0.2in} \; \leq \; \sum_{j=1}^n ( R_{nj} + \Xi_{nj} ) \leq
C  \sum_{j=1}^n \lambda_{nj}^2 r_{nj}  (1+r_{nj})  \leq
 C \Big(\max_{1\leq j \leq n} \lambda_{nj} \Big)\Big( \sum_{j=1}^n \lambda_{nj} r_{nj}  +
\sum_{j=1}^n \lambda_{nj} r_{nj} ^2 \Big) \rightarrow 0 .
\end{eqnarray*}
the last conclusion following from Assumptions A.1 and A.3.
Using the above in conjunction with Proposition \ref{prop:empirical_L} we obtain
\begin{eqnarray*}
\lim_{n\rightarrow \infty} \log \bE  e^{i\sum_{k=1}^m \theta_k X_n (t_{k})}
&  = & \sum_{u=1}^m \int_0^\infty (e^{ix\theta_u}-1)\frac{1}{\mu}\nu(dx) \\
&& \hspace{-1.6in} +\sum_{1\leq u<v\leq m}\left[ 1-e^{-\mu (t_{u}-t_{u-1})}\right]
e^{-\mu (t_v-t_u)}\left[ 1-e^{-\mu (t_{v+1}-t_v)}\right]
\int_0^\infty \left( e^{ix(\theta_u + \cdots + \theta_v)}-1\right) \frac{1}{\mu}\nu(dx).
\end{eqnarray*}
Absorbing the first sum into the second sum the right hand side above can be written as
\[
 \sum_{1\leq u \leq v\leq m}\left[ 1-e^{-\mu (t_u - t_{u-1})}\right]
e^{-\mu (t_v-t_u)}\left[ 1-e^{-\mu (t_{v+1}-t_v)}\right]
\int_0^\infty \left( e^{ix(\theta_u + \cdots + \theta_v)}-1\right)\frac{1}{\mu}\nu(dx)
\]
which is the right hand side of (\ref{GOU}).

\paragraph{Part 2. Tightness.}

Having established the convergence of the finite-dimensional distributions
we need to also establish the tightness of the family of processes in order
to establish convergence in distribution and complete the proof. This is
done by using Theorem \ref{th:Billingsley} with $\beta =2$ and showing 
that inequality (\ref{Billingsley-ineq}) holds for $\alpha=1$.
\begin{eqnarray}     \nonumber
&&\hspace{-0.28in}\bE  \left[ \left( X_n (t_2 )-X_n (t_1 )\right)^2 \left(
X_n (t_3 )-X_n (t_2 )\right)^2 \right]  
\;=\; \bE \left[ \left( \sum_{k=1}^n  \zeta_{nk}(t_2)-\zeta_{nk}(t_1) \right)^2
\left( \sum_{k=1}^n \zeta_{nk}(t_3)-\zeta_{nk}(t_2) \right)^2 \right] \\
\label{threeterms}
&&\hspace{-0.2in}=\;\bE  \left[ \sum_{k=1}^n \left( \zeta_{nk}(t_2 )-\zeta_{nk}(t_1)\right)^2
\left( \zeta_{nk}(t_3)-\zeta_{nk}(t_2)\right)^2 \right] \\  \nonumber
&&\hspace{-0.05in}
+\bE  \left[ \sum_{k\neq m}\left( \zeta_{nk}(t_2 )-\zeta_{nk}(t_1 )\right)^2
\left( \zeta_{nm}(t_3) - \zeta_{nm}(t_2 )\right)^2 \right] \\  \nonumber
&&\hspace{-0.05in}+2 \bE \hspace{-0.03in}
 \left[ \sum_{k\neq m} \hspace{-0.05in}\left( \zeta_{nk}(t_2)-\zeta_{nk}(t_1)\right) \hspace{-0.02in}
\left( \zeta_{nk}(t_3)-\zeta_{nk}(t_2)\right)\hspace{-0.02in}
\left(\zeta_{nm}(t_2 )-\zeta_{nm}(t_1)\right)\hspace{-0.02in}
\left( \zeta_{nm}(t_3)-\zeta_{nm}(t_2)\right) \right] .
\end{eqnarray}
(The above includes only the nonzero terms. For instance,
\[
\bE[\left( \zeta_{nj}(t_2)-\zeta_{nj}(t_1)\right) \hspace{-0.02in}
\left( \zeta_{nk}(t_3)-\zeta_{nk}(t_2)\right)\hspace{-0.02in}
\left(\zeta_{nl}(t_2 )-\zeta_{nl}(t_1)\right)\hspace{-0.02in} \left( \zeta_{nm}(t_3)-\zeta_{nm}(t_2)\right)]=0
\]
for terms for with $j \ne k$ since the sources are independent.)
For the first term in (\ref{threeterms}) we have, in view of (\ref{ineqL1}) of Lemma \ref{lemma:ONOFF}
in the Appendix, that
\begin{equation*}
\bE \left[ \sum_{k=1}^n \left( \zeta_{nk}(t_2)-\zeta_{nk}(t_1)\right)^2
\left(\zeta_{nk}(t_3)-\zeta_{nk}(t_2)\right)^2 \right] \; \leq \; \frac{\mu (t_3-t_1)^2}{4}\sum_{k=1}^n
\lambda_{nk}r_{nk}^4  .
\end{equation*}
For the second term in (\ref{threeterms}) we have that
\[
 \sum_{k\neq m} \bE \left[ \left( \zeta_{nk}(t_2)-\zeta_{nk}(t_1 )\right)^2 \right] \,
\bE \left[ \left( \zeta_{nm}(t_3 )-\zeta_{nm}(t_2 )\right)^2 \right]
\]
since the sources are independent. Also, from (\ref{ineqL3}),
\begin{eqnarray*}
\bE \left( \zeta_{nk}(t_2)-\zeta_{nk}(t_1 )\right)^2 \, \bE \left( \zeta_{nm}(t_3 )-\zeta_{nm}(t_2 )\right)^2
\leq && \hspace{-0.22in} 4\mu^2(t_3-t_2)(t_2-t_1) \frac{\lambda_{nk} r_{nk}^2}{\lambda_{nk}+\mu} \,
\frac{\lambda_{nm} r_{nm}^2}{\lambda_{nm}+\mu} \\
\leq &&\hspace{-0.22in} \mu^2(t_3-t_1)^2 \frac{\lambda_{nk} r_{nk}^2}{\lambda_{nk}+\mu} \,
\frac{\lambda_{nm} r_{nm}^2}{\lambda_{nm}+\mu}
\end{eqnarray*}
and thus second term in (\ref{threeterms}) is bounded by
\begin{eqnarray*}
\bE \left[ \sum_{k\neq m}\left( \zeta_{nk}(t_2)-\zeta_{nk}(t_1 )\right)^2
\left( \zeta_{nm}(t_3 )-\zeta_{nm}(t_2 )\right)^2 \right] &\leq &\mu^2(t_3-t_1)^2
\left( \sum_{k=1}^n \frac{\lambda_{nk} r_{nk}^2}{\lambda_{nk}+\mu} \right)^2 \\
&\leq& (t_3-t_1)^2 \left( \sum_{k=1}^n \lambda_{nk} r_{nk}^2 \right)^2 .
\end{eqnarray*}
For the third term of (\ref{threeterms}) we note that, because of the independence of sources, the fact
that $k \ne m$, and (\ref{ineqL2})
\begin{eqnarray*}
&& \Big| \bE \left[ \left( \zeta_{nk}(t_2)-\zeta_{nk}(t_1)\right)
\left( \zeta_{nk}(t_3)-\zeta_{nk}(t_2)\right)
\left(\zeta_{nm}(t_2 )-\zeta_{nm}(t_1)\right) \left( \zeta_{nm}(t_3)-\zeta_{nm}(t_2)\right) \right] \Big|  \\
&& \hspace{0.5in} \leq \; \frac{1}{16} \mu^2 \lambda_{nk} \lambda_{nm} r_{nk}^2 r_{nm}^2
(t_3-t_1)^4.
\end{eqnarray*}
Hence,
\begin{eqnarray*}
&&\hspace{-0.25in} 2 \bE \hspace{-0.03in} \left[ \sum_{k\neq m}
 \hspace{-0.03in}\left( \zeta_{nk}(t_2 )-\zeta_{nk}(t_1 )\right)
 \hspace{-0.03in}\left( \zeta_{nk}(t_3 )-\zeta_{nk}(t_2 )\right) \hspace{-0.03in}
\left( \zeta_{nm}(t_2 )-\zeta_{nm}(t_1 )\right) \hspace{-0.03in}
\left( \zeta_{nm}(t_3 )-\zeta_{nm}(t_2 )\right) \right] \\
&&\hspace{0.15in} \;\leq \;
\frac{\mu^2(t_3-t_1)^4}{8} \sum_{k \ne m} \lambda_{nk}r_{nk}^2 \lambda_{nm}r_{nm}^2
\; \leq \; \frac{\mu^2(t_3-t_1)^4}{8} \left( \sum_{k =1}^n  \lambda_{nk}r_{nk}^2 \right)^2.
\end{eqnarray*}
Therefore,
\begin{eqnarray*}
\bE  \left[ \left( X_n (t_2) - X_n (t_1)\right)^2
\left(X_n (t_3) - X_n (t_2)\right)^2 \right] &\leq &
\frac{\mu (t_3-t_1)^2}{4}\sum_{k=1}^n \lambda_{nk}r_{nk}^4  \\
&& \hspace{-2in} + (t_3-t_1)^2
\left( \sum_{k=1}^n \lambda_{nk} r_{nk}^2 \right)^2
 + \frac{\mu^2(t_3-t_1)^4}{8} \left( \sum_{k =1}^n  \lambda_{nk}r_{nk}^2 \right)^2.
\end{eqnarray*}
Assumption A.3 implies that
${\displaystyle \limsup_{n\rightarrow \infty}} \sum_{k=1}^n \lambda_{nk}r_{nk}^4  \leq C$ and
${\displaystyle \limsup_{n\rightarrow \infty}} \sum_{k=1}^n \lambda_{nk}r_{nk}^2 \leq C$
and hence, for some $K>0$,
\[
\bE \left[ \left( X_n (t_2) - X_n (t_1)\right)^2
\left(X_n (t_3) - X_n (t_2)\right)^2 \right] \leq K(t_3-t_1)^2 \;\;\; \mbox{for all $0<t_1<t_3 <T$.}
\]

Finally there remains to establish (\ref{tight-T}) and to this end it suffices to show that the limit process
$X$ satisfies the condition
$\bP(|X(T) - X(T-\delta)| > \epsilon) \rightarrow 0$ as $\delta \rightarrow 0$ for any
$\epsilon >0$. From Markov's inequality we have
\begin{equation} \label{Markov-psi}
\bP(|X(T) - X(T-\delta)| > \epsilon) \; \leq \;
\frac{\bE \left(X(T) - X(T-\delta)\right)^2}{\epsilon^2} \; \leq \;
\frac{1}{\epsilon^2} 2 H(\delta) \int_0^\infty x^2 \frac{1}{\mu} \nu(dx)
\end{equation}
where in the last inequality we have used (\ref{Difference}), (\ref{second-moment}),
together with the fact that 
$\psi(\theta) = \int_0^\infty  \left( e^{i\theta x} - 1 \right) \frac{1}{\mu}\nu(dx)$
and $\int_0^\infty x^2 \nu(dx)<\infty$ by virtue of Assumption A.5. Also, 
$H(\delta) = 1-e^{-\mu \delta}$. Since
$\lim_{\delta \rightarrow 0} H(\delta) =0$ from  (\ref{Markov-psi}) we obtain (\ref{tight-T}).
\end{proof}
The moment condition in Assumptions A.3 and A.5 are necessary in
establishing the tightness of the family of processes $\{X_n(t)\}$ because
of inequality (\ref{Billingsley-ineq}) in Theorem \ref{th:Billingsley}. In particular,
A.5 restricts the possible limit processes $\{X(t)\}$ to
those with finite first and second moments. It may be possible to circumvent this
restriction by using an alternative to (\ref{Billingsley-ineq}) which does not involve moments
(see for instance Theorem 13.5 in \cite[p.142]{Billingsley_M}) and also modify the proof
of part 2 of Theorem \ref{th:array-convergence} in order to relax these requirements.
We will not pursue this direction here.


\section{Further examples of GCID processes}

Here we discuss further the class of Gaussian GCIP processes and give a few additional
examples as well as examples of non-Gaussian processes.

\begin{proposition} \label{prop:gaussian}
Suppose that $\{X_t;t\in \bR \}$ is a GCID process with
$\psi(\theta) = i\mu \theta - \frac{1}{2}\sigma^2 \theta^2 $ and given $H$
satisfying the conditions of Theorem \ref{basic-theorem-H}.
Then $\{X_t\}$ is a stationary Gaussian process with covariance function
\begin{equation} \label{Gaussian-covariance}
\mbox{\sf Cov}(X_{t_i}, X_{t_j}) \;=\; \sigma^2 \left( 1 - H(t_j-t_i) \right), \;\;\;   t_i \leq t_j.
\end{equation}
If in addition $H$ satisfies
\begin{equation} \label{H0gaussian}
\displaystyle \lim_{h\downarrow 0} H(h) \, |\log h|^{1+\beta}\;=\;0 \;\;\;
\mbox{ for some } \beta >0
\end{equation}
then $\{X_t\}$ has a.s.\ continuous sample paths.
\end{proposition}
\begin{proof}
From Theorem \ref{basic-theorem-H} and Lemma \ref{le:drift}
\begin{eqnarray}
\phi_n(\theta_1,\ldots,\theta_n;t_1,\ldots,t_n) \,=\, i \mu (\theta_1+\cdots+\theta_n) -
\frac{\sigma^2 }{2} \sum_{1 \leq i \leq j \leq n} (\theta_i+\cdots +\theta_j)^2 \, a_{ij}
\end{eqnarray}
with $a_{ij} \;=\; \left[ H(t_j-t_{i-1}) - H(t_j-t_i ) - H(t_{j+1}-t_{i-1}) + H(t_{j+1}-t_i )\right]$.
From (\ref{ba}), (\ref{ab}),
\[
 \sum_{1 \leq i \leq j \leq n} (\theta_i+\cdots +\theta_j)^2 \, a_{ij} \;=\;
 \sum_{1\leq i \leq n,\; 1 \leq j \leq n} \theta_i \theta_j b_{ij}
 \]
with
\begin{eqnarray*}
b_{ij} &=& \sum_{k=1}^{i} \sum_{l=j}^n
\left[ H(t_j-t_{i-1}) - H(t_j-t_i ) - H(t_{j+1}-t_{i-1}) + H(t_{j+1}-t_i )\right] \\
&=& 1 - H(t_j-t_i), \;\; \mbox{ when } i\leq j
\end{eqnarray*}
and $b_{ji}=b_{ij}$,
where in the above we use the telescopic nature of the sums and the convention $t_0 = -\infty$,
$t_{n+1} = +\infty$, $H(0)=0$, $H(+\infty) =1$. Thus
\begin{eqnarray}
\log \bE[e^{i \sum_{k=1}^n \theta_k X_{t_k}}] \;=\; i \mu \sum_{k=1}^n \theta_k  -
\frac{\sigma^2}{2} \sum_{i=k}^n \theta_k^2
- \sigma^2  \sum_{1\leq j < k \leq n} \theta_j \theta_k \left[1- H(t_k -t_j) \right].
\end{eqnarray}
Thus $\{X_t; \;t\in \bR \}$ is a stationary Gaussian process with covariance given by
(\ref{Gaussian-covariance}).

Let
\begin{equation} \label{pu}
p(u):= \sup_{0\leq t \leq u} \bE(X_t-X_0)^2 .
\end{equation}
A  necessary and sufficient condition for a stationary Gaussian process to have
continuous paths with probability 1 is
\[
\int_0^\delta (-\log u)^{1/2}dp(u) \;<\; \infty
\]
for some $\delta >0$. (See \cite[p.14-15]{Adler}.) A sufficient condition, easier
to check, is
\begin{equation} \label{pu2}
p(u) \; \leq \; \frac{C}{|\log u|^{1+\beta}}
\end{equation}
for some $C>0$ and $\beta>0$. Using (\ref{second-moment}) equation
(\ref{pu}) gives $p(u) = \sup_{0\leq t \leq u} \sigma^2 H(t)
= \sigma^2 H(u)$. Therefore (\ref{H0gaussian}) implies that (\ref{pu2}) holds
and hence that the sample paths of the process are a.s. continuous.
This completes the proof.
\end{proof}

The classical Ornstein--Uhlenbeck process is a Gaussian GCID process with
$H(t) = 1- e^{-\alpha t}$ (with $\alpha > 0$), satisfies (\ref{H0gaussian})
since $\lim_{h\downarrow 0}H(h)/h = \alpha$, and therefore has (as is well known)
continuous sample paths w.p.\ 1. On the other hand its GCID counterpart with
Poisson marginals,
the $M/M/\infty $ process, has sample paths that w.p. 1 are piecewise constant and
have jumps of size 1. Thus it is clear that path behavior depends both on the nature
of the marginal distributions and on the type of correlation structure which is
determined by $H$.

Another instance of a Gaussian GCID process which also satisfies the condition of
Proposition \ref{prop:gaussian} and has therefore a.s.\ continuous
sample paths is the following generalization of the Slepian-Shepp process
(see Slepian \cite{Slepian} and  Shepp \cite{Shepp1}) which
we discuss in the following
\begin{example} \label{ex:min}
{\rm Let  $H(x)=\min (x^\alpha,1)=x^\alpha \wedge 1$ where $\alpha \in (0,1]$ and
$\psi (\theta) = - \frac{1}{2}\sigma^2 \theta^2 $. Then (\ref{main}) gives
\begin{eqnarray}  \label{example2}
&& \hspace{-0.35in} \log \phi_n (\theta_1 ,\ldots ,\theta_n ;t_1 ,\ldots ,t_n )
\;=\; -\frac{1}{2}\sigma^2  \\
&& \hspace{-0.3in} \times \hspace{-0.15in} \sum_{1\leq i\leq j\leq n} \hspace{-0.1in}
(\theta_i +\cdots +\theta_j)^2 \Big[ \big[(t_j-t_{i-1})^\alpha \wedge 1\big]-
\big[(t_j-t_i )^\alpha \wedge 1\big]- \big[(t_{j+1}-t_{i-1})^\alpha \wedge 1\big]
+\big[(t_{j+1}-t_i )^\alpha \wedge 1\big]\Big]  . \notag
\end{eqnarray}
Then from Proposition \ref{prop:gaussian}, (\ref{example2}) can be written as
\begin{equation*}
-\frac{1}{2}\sigma^2 \left\{ \sum_{i=1}^n \theta_i ^2
+2\sum_{1\leq i < j\leq n}\theta_i \theta_j
\Big[ 1-\big[\left( t_j-t_i \right)^\alpha  \wedge 1\big]\Big] \right\}
\end{equation*}
and thus it represents a stationary Gaussian process with covariance function
$R(\tau ):=\text{\sf \small Cov}(X_t ,X_{t+\tau })$ $ = \sigma^2 (1- \tau^\alpha )^+$,
$\tau \geq 0$. This falls within the framework of Gaussian processes discussed in
Remark \ref{rem:gaussian}.
}
\end{example}
When $\alpha =1$ the process of Example \ref{ex:min} is the Slepian-Shepp Gaussian
process. If in (\ref{example2}) we take $\psi(\theta)= \rho (e^{i\theta} - 1)$ we
obtain the process with joint distributions
\begin{eqnarray}  \label{MMinfSS}
&& \hspace{-0.35in} \log \phi_n (\theta_1 ,\ldots ,\theta_n ;t_1 ,\ldots ,t_n )
\;=\;\hspace{-0.15in} \sum_{1\leq i\leq j\leq n} \hspace{-0.1in}
\rho\big(e^{i(\theta_i +\cdots +\theta_j)}-1\big) \\
&& \hspace{0.1in}   \times  \Big[ \big[(t_j-t_{i-1})^\alpha \wedge 1\big]-
\big[(t_j-t_i )^\alpha \wedge 1\big]- \big[(t_{j+1}-t_{i-1})^\alpha \wedge 1\big]
+\big[(t_{j+1}-t_i )^\alpha \wedge 1\big]\Big]  . \notag
\end{eqnarray}
When $\alpha=1$ the above is the $M/D/\infty$ process
with service times equal to 1. Indeed, we can check this by referring to section \ref{sec:mg1}
and taking the service time distribution in the $M/GI/\infty$ model to be
$G(x) = \ind(x \geq 1)$ which implies that $G_I(x) = x\wedge 1$ for $x>0$.

\begin{example}
{\rm We take again $H(x)=x^\alpha\wedge 1$ with $\alpha \in (0,1]$. However,
this time we choose $\psi(\theta)=-\log (1-i\theta)$. Then
\begin{eqnarray}  \label{GE2}
\log \phi_n(\theta_1,\ldots ,\theta_n;t_1,\ldots ,t_n) &=&-\sum_{1\leq i\leq
j\leq n}\log [1-i(\theta_i+\cdots +\theta_j)] \\
&& \hspace{-1.2in} \times \; \left[ (t_j-t_{i-1})^\alpha\wedge 1-(t_j-t_i)^\alpha\wedge
1-(t_{j+1}-t_{i-1})^\alpha\wedge 1+(t_{j+1}-t_i)^\alpha\wedge 1\right] .  \notag
\end{eqnarray}
This defines a process which has the same covariance structure as the
Gaussian process of Example \ref{ex:min} but has Gamma marginals. In particular,
in the case $\alpha=1$ and assuming $t_n-t_1\leq 1$, the right hand side of
(\ref{GE2}) reduces to
\begin{eqnarray*}
&& -\sum_{j=1}^{n-1}\log [1-i(\theta_1+\cdots +\theta_j)](t_{j+1}-t_j)\;
-\;\sum_{i=2}^n \log [1-i(\theta_i+\cdots +\theta_n)](t_i-t_{i-1})\; \\
&& \hspace{2in} -\;\log [1-i(\theta_1+\cdots +\theta_n)]\left[ 1-( t_n-t_1) \right] .
\end{eqnarray*}
The distribution
of $X_{t+h}-X_t $ can be obtained from (\ref{Difference}). Its
characteristic function is
\begin{equation*}
\exp \Big( H(h)\big[ -\log (1-i\theta )-\log (1+i\theta )\big] \Big)
=\left( \frac{1}{1+\theta^2 }\right) ^{H(h)}.
\end{equation*}
This is in fact the difference of two independent Gamma--distributed random
variables.

In particular, note that
\begin{eqnarray}  \nonumber
\phi(\theta_1,\theta_2) &=& \bE  \left[ e^{i\theta_1 X_{t_1}
+ i\theta_2X_{t_2}}\right]  \\  \label{expslepian}
&=& \left( \frac{1}{1-i (\theta_1+\theta_2)}\right)^{1-H(t_2-t_1)}
\left( \frac{1}{1-i\theta_1}\right)^{H(t_2-t_1)} \left(\frac{1}{1-i\theta_2}\right)^{H(t_2-t_1)}.
\end{eqnarray}
Thus (\ref{GE2}) defines a type of multidimensional exponential distribution. If the
distribution $H$ has finite support, say $H(t)=1$ for $t\geq \tau$, then $X_{t_1}$,
$X_{t_2}$ are independent for $t_2-t_1>\tau$ as can be seen from (\ref{expslepian}).}
\end{example}

\begin{exmpl}{\rm This is an illustration of Proposition \ref{prop:superposition}.
If in (\ref{main}) we take $\psi(\theta) = \frac{1}{2}\sigma^2 \theta^2$ and
$H(t):=  1 - \sum_{k=1}^K p_k e^{-\alpha_k t}$ where $\sigma >0$, $p_k>0$,
$\alpha_k >0$, $k=1,\ldots,K$, and $\sum_{k=1}^K p_k =1$, then the resulting process
is the superposition of $K$ independent (zero mean) stationary Ornstein-Uhlenbeck
processes, $\{Y^{(k)}_t;t\geq 0\}$, $k=1,\ldots,K$ where
$\bE[Y_{t+s}^{(k)} Y_{s}^{(k)}] = \sigma^2 p_k \, e^{-\alpha_k t}$,
$t>0$, $k=1,\ldots,K$.
}
\end{exmpl}

Unlike proposition 1 where the question of path continuity can be addressed based
on known results about Gaussian processes, in general the question of path continuity
for GCIDs not corresponding to already known processes is more difficult to settle.
It will be addressed in a future paper. A positive answer to the question
of path regularity may be obtained by a straightforward application of the Kolmogorov-Chentsov
criterion according to which, if a process $\{X_t\}$ satisfies the inequality
$\bE\big[ |X_{t_2}-X_{t_1} |^\alpha \, | X_{t_3}-X_{t_2}|^\alpha \big] \,\leq \,
|t_3-t_1|^{1+\beta}$ for $0<t_1<t_2<t_3$ and $\alpha>0$, $\beta>0$ when
$t_3$ is sufficiently small then w.p. 1 it possesses paths that have no discontinuities of the
second kind \cite{Chentsov}.
Assuming that the GCID process have finite fourth moments and furthermore, assuming
for the sake of simplicity that $\psi'(0)=\bE X_t =0$ we may obtain from (\ref{main})
\[
\bE\big[ (X_{t_2}-X_{t_1} )^2 \, ( X_{t_3}-X_{t_2})^2 \big] \,= \, \psi^{(4)}(0) a_0
+ \psi''(0)^2 \, \big[(a_0+a_1)(a_0+a_2) + 2 a_0^2\big]
\]
with $a_1:= H(t_3-t_1)+H(t_2-t_1)+H(t_3-t_2)$, $a_2:= H(t_3-t_1)+H(t_3-t_2)+H(t_2-t_1)$,
and $a_0:= H(t_2-t_1)+H(t_3-t_2)-H(t_3-t_1)$. In the special case where $H$ is twice
continuously differentiable in a neighborhood of the origin with $H'(0)= \ell>0$ and
$H''(0) =b<0$, there exists $h$ such that whenever $0<t<h$,
$H(t)=\ell t + \frac{1}{2}bt^2+ o(h^2)$. Then
\[
\bE\big[ (X_{t_2}-X_{t_1} )^2 \, ( X_{t_3}-X_{t_2})^2 \big] \,= \, \left(\frac{1}{8} b \psi^{(4)}(0)
+ 4 \ell^2\psi''(0)\right)(t_3-t_1)^2 + o(h^2)
\]
and thus Kolmogorov's criterion is satisfied with $\alpha=2$ and $\beta=1$.

\appendix
\section{Appendix}
\subsection{Bounds for ON/OFF Markovian Processes}

\begin{lemma}   \label{lemma:ONOFF}
\bigskip Let $\{\zeta_t \}$ be a Markovian ON/OFF source taking the values $0$ and $r$
assumed stationary, and suppose that $u<t<s$. Then,
\begin{eqnarray}  \label{ineqL1}
\mathbb{E}\left[ (\zeta_t -\zeta_u)^2 (\zeta_s-\zeta_t)^2 \right] &\leq &
\frac{r^4}{4}\lambda \mu (s-u)^2 , \\   \label{ineqL2}
\left\vert \mathbb{E}\left[ (\zeta_u -\zeta_t )(\zeta_t -\zeta_s) \right] \right\vert &\leq&
\frac{r^2}{4}\lambda \mu (u-s)^2 ,\\ \label{ineqL3}
\mathbb{E} \left[ (\zeta_t -\zeta_u)^2 \right] &\leq & \frac{2\lambda \mu r^2}{\lambda+\mu} (t-u).
\end{eqnarray}
\end{lemma}
\begin{proof}
The product $(\zeta_t -\zeta_u)(\zeta_s-\zeta_t)$ is non zero only when
$\zeta_u = 0,\zeta_t =r,\zeta_s=0$, and $\zeta_u=r,\zeta_t=0,\zeta_s = r$.
In both cases the product is equal to $-r^2$. Then
\begin{eqnarray*}
\mathbb{E} \left[ (\zeta_t -\zeta_u )^2 (\zeta_s -\zeta_t )^2 \right] &=&
r^4\mathbb{P}(\zeta_u =0)\mathbb{P}(\zeta_t =r|\zeta_u =0)\mathbb{P}(\zeta_s =0|\zeta_t =r) \\
&& \hspace{0.3in} +\;
r^4\mathbb{P}(\zeta_{u}=r)\mathbb{P}(\zeta_t =0|\zeta_u =r)\mathbb{P}(\zeta_s =r|\zeta_t =0)
\end{eqnarray*}
and thus
\begin{equation*}
\mathbb{E} \left[ (\zeta_t -\zeta_u )^2 (\zeta_s -\zeta_t )^2 \right] =
\frac{\lambda \mu r^4 }{(\lambda +\mu )^2}(1-e^{-(\lambda +\mu)(t-u)})
(1-e^{-(\lambda +\mu )(s-t)}).
\end{equation*}
Taking into account the fact that
\begin{equation}  \label{ineqL4}
0\leq 1-e^{-(\lambda +\mu )x}\leq (\lambda +\mu )x \;\; \mbox{ when $x>0$ }
\end{equation}
we see that
\begin{equation*}
\mathbb{E} \left[ (\zeta_t -\zeta_u )^2 (\zeta_s -\zeta_t )^2 \right] \leq
\frac{r^4\lambda \mu }{(\lambda +\mu )^2 }(\lambda +\mu)^2 (t-u)(s-t) \leq
r^4\lambda \mu \frac{1}{4}(s-u)^2
\end{equation*}
where the last inequality follows from the fact that $(t-u)(s-t)\leq \frac{1}{4}(s-u)^2$.

Inequality (\ref{ineqL2}) is established by taking again into account the fact
that $\zeta_u -\zeta_t $ and $\zeta_t -\zeta_s $ take values in the set $\{-r,0,r\}$. Then,
by Jensen's inequality,
\begin{equation*}
\big\vert \mathbb{E} \left[ (\zeta_u -\zeta_t )(\zeta_t -\zeta_s )\right]
\big\vert \leq \mathbb{E} \left[ \left\vert \zeta_u -\zeta_t \right\vert
\left\vert \zeta_t -\zeta_s \right\vert \right] =r^{-2}
\mathbb{E} \left[(\zeta_u-\zeta_t)^2 (\zeta_t -\zeta_s)^2 \right]
\end{equation*}
(the last equality following from the fact that the process $\{\zeta_t\}$ takes values
in the set $\{-r,0,r\}$) and hence we can use (\ref{ineqL1}) to establish (\ref{ineqL2}).

The last inequality can be established by the same type of argument as the first by noting
that
\begin{eqnarray*}
\mathbb{E} \left[ (\zeta_t -\zeta_u )^2 \right] &=&
r^2\mathbb{P}(\zeta_u =0)\mathbb{P}(\zeta_t =r|\zeta_u =0)
+r^2\;\mathbb{P}(\zeta_{u}=1)\mathbb{P}(\zeta_t =0|\zeta_u =r) \\
&=& \frac{2\lambda \mu r^2}{(\lambda+\mu)^2}\,
\left( 1 -  e^{-(\lambda+\mu)(t-u)} \right).
\end{eqnarray*}
Using (\ref{ineqL4}) we establish (\ref{ineqL3}).
\end{proof}

\begin{proposition}   \label{prop:boundB}
There is $M>0$ such that, for all $k \in \{2,\ldots,n\}$ the quantity $L_{nj}(k)$ 
defined in (\ref{fact3}), satisfies the bound
\begin{equation} \label{estimate-bound}
0 \leq \bE  [\xi_{nj}(t_{l_1}) \cdots \xi_{nj}(t_{l_k})] -\frac{\lambda_{nj}}{\mu}  
e^{-\mu(t_{l_k}-t_{l_1})} \leq M \lambda^2_{nj}.
\end{equation}
Also, the quantity
\begin{eqnarray}  \nonumber
	R_{nj}&:=& \;\sum_{k=2}^m \hspace{0.1in}\sum_{1\leq l_1 < l_2 <\cdots < l_k \leq m}
	\hspace{-0.05in}\left( \bE  \left[ \xi_{l_1}\cdots \xi_{l_k}\right]
	-\frac{\lambda_{nj}}{\mu}  e^{-\mu(t_{l_k}-t_{l_1})} \right)
(e^{ir\theta_{l_1}}-1)\cdots (e^{ir\theta_{l_k}}-1) \\    \label{ccf2}
&& + \left( \pi_{nj} - \frac{\lambda_{nj}}{\mu} \right) \sum_{1\leq l_1 \leq m} 
\big(e^{ir\theta_{l_1} }-1 \big)
\end{eqnarray}
satisfies the inequality
\begin{equation}  \label{ineqccf2}
|R_{nj}| \leq  \lambda_{nj}^2 r_{nj} K \;\;\; \mbox{ for some $K>0$.}
\end{equation}
\end{proposition}
\begin{proof}
Expanding the product in (\ref{estimate}) we may obtain the bound
$\bE  [\xi_{nj}(t_{l_1}) \cdots \xi_{nj}(t_{l_k})] - \pi_{nj} e^{-\alpha_{nj}(t_{l_k}-t_{l_1})}
\leq  \pi_{nj}^2 M_1$ for some positive $M_1$. Hence, taking into account (\ref{notation}),
\[
\bE  [\xi_{nj}(t_{l_1}) \cdots \xi_{nj}(t_{l_k})] - \frac{\lambda_{nj}}{\mu}
e^{-\mu( t_{l_k}-t_{l_1})} \leq M_1 \pi_{nj}^2 - \pi_{nj} e^{-\mu(t_{l_k}-t_{l_1})}
\big( 1- e^{\lambda_{nj}(t_{l_k}-t_{l_1})} \big)
\]
Using the inequalities
$\pi_{nj}=\frac{\lambda_{nj}}{\mu+ \lambda_{nj}} \leq \frac{\lambda_{nj}}{\mu}$
and $0 \leq 1- e^{-\lambda_{nj} (t_{l_k}-t_{l_1})} \leq \lambda_{nj}(t_{l_k}-t_{l_1})$,
elementary considerations give the bound (\ref{estimate-bound}).

Use the triangular inequality in (\ref{ccf2}) and take into account (\ref{estimate-bound}).
Also use the inequality $\left| e^{i\theta r} - 1 \right| \leq 2 \wedge |\theta r|$  to write
$\left| (e^{ir\theta_{l_1}}-1)\cdots (e^{ir\theta_{l_k}}-1)\right| \leq r| \theta_{l_1}|  2^{k-1}$.
Hence the modulus of the first term in the right hand side of (\ref{ccf2}) is bounded by
\begin{eqnarray*} 
 \sum_{k=2}^m \hspace{0.1in}
\sum_{1\leq l_1 < l_2 <\cdots < l_k \leq m} M \lambda_{nj}^2r| \theta_{l_1}|  2^{k-1} \;=\;
\frac{| \theta_{l_1}|}{2}M \lambda_{nj}^2r \sum_{k=2}^m {m \choose k} 2^k   \leq
 \lambda_{nj}^2 r_{nj} \frac{| \theta_{l_1}|}{2}M 3^m
\end{eqnarray*}
where in the last inequality we have used the binomial theorem.
The modulus of the second term of (\ref{ccf2})
is bounded by
\[
\frac{\lambda_{nj}^2}{\mu^2} r \sum_{l_1=1}^m |\theta_{l_1}|.
\]
Combining these two terms (and recalling that $r$ is shorthand for $r_{nj}$) establishes 
the bound (\ref{ineqccf2}).
\end{proof}

\subsection{A algebraic identity}
\begin{proposition} \label{prop:ai} For any $m \in \bN$, $t_1<t_2<\cdots<t_m$, and
$\alpha,r$, $\theta_1,\ldots, \theta_m \in \bR$
\begin{eqnarray}  \label{Algebraic-Identity}
 \sum_{k=2}^m \hspace{-0.6in}\sum_{\hspace{0.7in} 1\leq l_1 <l_2 <\cdots <l_k \leq m}
\hspace{-0.55in} e^{-\alpha (t_{l_k}-t_{l_1})}(e^{ir\theta_{l_1 }}-1)
\cdots (e^{ir\theta_{l_k}}-1) && \\  \nonumber
&& \hspace{-3in} =
\sum_{1\leq u < v \leq m} \left( e^{ir_{nj}(\theta_u +\cdots +\theta_v)}-1\right)
\left[ 1-e^{-\alpha (t_{u}-t_{u-1})}\right] e^{-\alpha (t_v-t_u)}
\left[ 1-e^{-\alpha (t_{v+1}-t_{v})}\right]
\end{eqnarray}
\end{proposition}
\begin{proof}
The left hand side of (\ref{Algebraic-Identity}) can be writen as
\begin{equation}   \label{cf4}
\sum_{1\leq u<v\leq m}e^{-\alpha (t_v-t_u)}(e^{ir\theta_u}-1)(e^{ir\theta_v}-1)
\left(1+\sum_{k=1}^{v-u-1}\sum_{u<l_1 <\cdots <l_k<v}
(e^{ir\theta_{l_1}}-1)\cdots (e^{ir\theta_{l_k}}-1)\right).
\end{equation}
Using the identity
\[
\prod_{j=1}^n (1+x_j) \;=\; 1+ \sum_{k=1}^n \hspace{-0.6in}
\sum_{\hspace{0.7in} 1 \leq j_1 < j_2 < \cdots < j_k \leq n} 
\hspace{-0.55in}x_{j_1}x_{j_2}\cdots x_{j_k}
\]
with $x_j = e^{ir\theta_j} - 1$, $j=1,\ldots,n$, and $n=v-u-1$ in the term inside the 
parentheses in (\ref{cf4})
we can rewrite it as
\begin{equation}   \label{inclusion_exclusion_2}
1+\sum_{k=1}^{v-u-1}\sum_{u<l_1 <\cdots <l_k<v}(e^{ir\theta_{l_1}}-1)\cdots 
(e^{ir\theta_{l_k}}-1) \;=\; e^{ir(\theta_{u+1}+\theta_{u+2}+\cdots +\theta_{v-1})}.
\end{equation}
Thus substituting (\ref{inclusion_exclusion_2}) in (\ref{cf4}) and taking into account that
\begin{eqnarray*}
&&(e^{ir\theta_u}-1)(e^{ir\theta_v}-1)e^{ir(\theta_{u+1} + \cdots +\theta_{v-1})} \\
&&\hspace{1in} =e^{ir(\theta_u +\cdots +\theta_v)}-e^{ir(\theta_{u+1}+\cdots +\theta_v)}
-e^{ir(\theta_{u}+\cdots +\theta_{v-1})}+e^{ir(\theta_{u+1}+\cdots +\theta_{v-1})}
\end{eqnarray*}
we rewrite (\ref{cf4}) as
\begin{eqnarray}   \notag
&&\hspace{-0.3in}\sum_{1\leq u<v\leq m}e^{-\alpha (t_v-t_u)}
\left[ e^{ir(\theta_u+\cdots +\theta_v)}
-e^{ir(\theta_{u+1}+\cdots +\theta_v)}-e^{ir(\theta_u + \cdots + \theta_{v-1})}
+e^{ir(\theta_{u+1}+\cdots +\theta_{v-1})}\right]   \\   \notag
&&\hspace{-0.2in}=\sum_{1\leq u<v\leq m}e^{ir(\theta_u + \cdots + \theta_v)}
\left[e^{-\alpha (t_v-t_u)}-e^{-\alpha (t_{v+1}-t_u)}- e^{-\alpha (t_v-t_{u-1})}
+e^{-\alpha (t_{v+1}-t_{u-1})}\right]  \\  \notag
&&\hspace{-0.2in} =\sum_{1\leq u<v\leq m}e^{ir(\theta_u + \cdots +\theta_v)}
\left[1-e^{-\alpha (t_u-t_{u-1})}\right] e^{-\alpha (t_v-t_u)}\left[1-e^{-\alpha (t_{v+1}-t_v)}\right] .
\end{eqnarray}
This establishes (\ref{Algebraic-Identity}).
\end{proof}

\subsection{An example of a double array satisfying the assumptions of
Theorem \ref{th:array-convergence}}
Here we consider an example of a double array of stationary ON/OFF processes
$\{\zeta_{nj}(t)\}$ with corresponding parameter array $(\lambda_{nj},r_{nj})$ given by
given by $\lambda_{nj} = n^{-\alpha}$, $r_{nj} = b^{jn^{-\alpha}}$, $j=1,\ldots n$,
$n=1,2,\ldots,$ where $\alpha$ and $b$ are in $(0,1)$. We will show that the
sequence of measures
$\nu_n(\cdot) = \sum_{j=1}^n n^{-\alpha} \delta_{ b^{jn^{-\alpha}}}(\cdot)$
converges weakly to
\begin{equation} \label{limit_measure}
\nu(dx)  \;=\; \frac{1}{x\, \log(1/b)} \ind(0 < x \leq 1).
\end{equation}
This is a $\sigma$-finite L\'evy measure which satisfies the moment conditions of A.5.

Next we show that A.4 is satisfied (thus establishing the weak convergence of the
measures $\{\nu_n\}$ to $\nu$). Note that $1 < r_{nj} = b^{jn^{-\alpha}} < 1$
and hence it suffices to take $x \in (0,1)$.
\[
r_{nj} = b^{jn^{-\alpha}} \geq x  \;\;\Leftrightarrow \;\;
j \leq j^* := \left \lfloor \frac{\log(1/x)}{\log(1/b)} n^\alpha \right \rfloor
\]
and thus
\begin{eqnarray}   \nonumber
\sum_{j=1}^n \lambda_{nj} \ind(r_{ nj} \geq x)
= n^{-\alpha} \left \lfloor \frac{\log(1/x)}{\log(1/b)} n^\alpha \right \rfloor
\rightarrow \frac{\log(1/x)}{\log(1/b)} &=&  \int_x^1 \frac{1}{y\,\log(1/b)} dy \\
&=&  \nu[x,\infty)  \hspace{0.2in} \mbox{as $n\rightarrow \infty$}.  \label{explA2}
\end{eqnarray}

Similarly, A.6 is satisfied. Indeed
\begin{eqnarray}   \nonumber
\sum_{j=1}^n \lambda_{nj} r_{nj} \ind(r_{ nj} \geq x)
= \sum_{j=1}^{j^*} n^{-\alpha} b^{jn^{-\alpha}} &= &
\frac{b^{n^{-\alpha}} - b^{\left\lfloor \frac{\log(1/x)}{\log(1/b)} n^\alpha
\right\rfloor n^{-\alpha}+n^{-\alpha}}}{n^{\alpha}\left(1- b^{n^{-\alpha}}\right)}
\rightarrow   \frac{1- b^{\frac{\log(1/x)}{\log(1/b)} }}{\log(1/b)} =
\frac{1-x}{\log(1/b)}\\
 &=& \int_x^\infty y \nu(dy)  \hspace{0.2in} \mbox{as $n\rightarrow \infty$}. \label{explA3}
\end{eqnarray}
We point out that, as a result of the above,
\begin{eqnarray}  \nonumber
\lim_{n\rightarrow \infty} \sum_{j=1}^n \lambda_{nj} \left(e^{i\theta r_{nj}}-1 \right)
&=& \lim_{n\rightarrow \infty} n^{-\alpha} \sum_{j=1}^n
\left(e^{i\theta b^{jn^{-\alpha}}}-1 \right)
= \int_0^1 \left(e^{i\theta x} -1 \right) \frac{\log(1/b)}{x}dx,\\
&& \hspace{1in} \;\;\; \mbox{ uniformly w.r.t. $\theta$ on compact intervals of $\bR$.}
\label{empirical_example}
\end{eqnarray}

Next, we check that A.2 holds. This argument is similar to that used to show that A.4 holds.
\[
r_{nj}\leq \epsilon \;\;\mbox{ or } \;\;
b^{jn^{-\alpha}} \leq \epsilon \Longleftrightarrow
j \geq j^*:= \left\lceil \frac{\log(1/\epsilon)}{\log(1/b)} n^\alpha \right\rceil
\]
and thus
\[
\sum_{j=1}^n n^{-\alpha} b^{jn^{-\alpha}}\ind( b^{jn^{-\alpha}} \leq \epsilon )
=n^{-\alpha} \sum_{j=j^*}^n b^{jn^{-\alpha}}
= \frac{b^{\left\lceil \frac{\log(1/\epsilon)}{\log(1/b)} n^\alpha \right\rceil n^{-\alpha}}-
b^{n^{1-\alpha}+n^{-\alpha}}}{n^{\alpha}\left(1- b^{n^{-\alpha}}\right)} \rightarrow
\frac{b^{ \frac{\log(1/\epsilon)}{\log(1/b)} }}{\log(1/b)} \hspace{0.2in} \mbox{ as $n \rightarrow \infty$.}
\]
Since $b \in (0,1)$, $\lim_{\epsilon \rightarrow 0} b^{ \frac{\log(1/\epsilon)}{\log(1/b)}} =0$ from which
we conclude that A.2 is satisfied.

A.1 is obviously satisfied. To check A.3 note that
\[
\sum_{j=1}^n \lambda_{nj} r_{nj}^p = n^{-\alpha} \sum_{j=1}^n b^{pjn^{-\alpha} }=
\frac{b^{pn^{-\alpha}}\left(1- b^{pn^{1-\alpha}}\right)}{n^{\alpha}\left(1- b^{pn^{-\alpha}}\right)}
\rightarrow \frac{1}{p \log(1/b)} \hspace{0.2in} \mbox{ as $n \rightarrow \infty$.}
\]
Since for any $p>0$ the limit exists and is finite the sequence is bounded and thus A.3 is
satisfied.

We now turn to the double array of stationary processes $\{\zeta_{nj}(t)\}$ corresponding
to the array of parameters $(\lambda_{nj}, r_{nj})$ in the framework of section
\ref{sec:ONOFFlim}.

Fixing $t$ and omitting it from the notion it is easy to see that the array or random
variables $\{\zeta_{nj}\}$ satisfies C.1 with respect to the L\'evy measure
$\mu^{-1} \nu(dx)$ on $(0,\infty)$. Indeed, from (\ref{explA2}) it follows immediately that
\begin{eqnarray} \nonumber
 \sum_{j=1}^n \bP(\zeta_{nj} \geq x) &=& \sum_{j=1}^n
\frac{n^{-\alpha} }{n^{-\alpha} +\mu} \ind(b^{jn^{-\alpha}}\leq \epsilon)
 \rightarrow \frac{1}{\mu} \nu[x,\infty)
\end{eqnarray}

We will see that Condition C.2 is satisfied with $\gamma=0$. Indeed
\begin{eqnarray} \nonumber
\sum_{j=1}^n \bE(\zeta_{nj} \ind(\zeta_{nj}\leq \epsilon)] &=&
\sum_{j=1}^n
\frac{\lambda_{nj} r_{nj}}{\lambda_{nj}+\mu} \ind(r_{nj}\leq \epsilon)
= \sum_{j=1}^n \frac{n^{-\alpha} b^{jn^{-\alpha}}}{n^{-\alpha}+\mu}
 \ind(b^{jn^{-\alpha}} \leq \epsilon)  \\     \label{empirical_example_3}
 &=&
  \sum_{n^\alpha C \leq j \leq n} \frac{ n^{-\alpha}  b^{jn^{-\alpha}}}{n^{-\alpha}+\mu}
 =  \frac{b^C - b^{n^{1-\alpha}}}{n^\alpha(1-b^{n^{-\alpha}})} \, \frac{1}{n^{-\alpha} + \mu}
\end{eqnarray}
where $C := \frac{\log \epsilon}{\log b}$ and we have used the fact that
$ b^{jn^{-\alpha}} \leq \epsilon \Leftrightarrow j \geq n^\alpha C$.
Since $0< \alpha, b < 1$ the last term in (\ref{empirical_example_3}), as
$n\rightarrow \infty$ converges to $b^C/\log(1/b)= \frac{\epsilon}{\log(1/b)}$.
Thus C.2 is satisfied in this situation.


\begin{thebibliography}{99}
\bibliographystyle{APT}

\bibitem{Adler}
Adler R. J. (1990). {\em An introduction to continuity, extrema, and related topics
for general Gaussian processes,} IMS Lecture Notes Monogr. Ser., 12.

\bibitem{Applebaum}
Applebaum D. (2009). {\em L\'evy Processes and Stochastic Calculus,} 2nd edn.
Cambridge University Press, Cambridge.

\bibitem{BNS} Barndorff-Nielsen O. E. and N. Shephard. (2001). Non-Gaussian
Ornstein--Uhlenbeck-based models and some of their uses in financial
economics, \emph{J. Roy. Statist. Soc. Ser. B Statist. Methodol.} \textbf{63}
(2), 167--241.

\bibitem{BN}
Barndorff-Nielsen O. E. (2011). Stationary infinitely divisible processes,
{\em Brazilian Journal of Probability and Statistics,} {\bf 25} (3), 294--322.

\bibitem{Billingsley_M} Billingsley, P. (1999). \textit{Convergence of Probability
Measures,} 2nd edn.\ John Wiley, New York.

\bibitem{Chentsov}
Chentsov N. N. (1956).
Weak convergence of stochastic processes whose trajectories
have no discontinuities of the second kind and the ``heuristic''
approach to the Kolmogorov-Smirnov tests. {\em Theory of
Probability \& Its Applications,} {\bf 1}, 1, 140--144.

\bibitem{FristedtGrey}
Fristedt B. and L. Gray. (1997) {\em A Modern Approach to Probability
Theory,} Birkh\"auser, Boston.

\bibitem{Hall}
Hall P. (1988). {\em Introduction to the Theory of Coverage Processes}
John Wiley \& Sons, New York.

\bibitem{HeathResnickSamorodnitsky}
Heath D., S. Resnick, and G. Samorodnitsky. (1998).
Heavy Tails and Long Range Dependence in On/Off
Processes and Associated Fluid Models, {\em Mathematics
of Operations Research,} {\bf 23} (1), 145-165.

\bibitem{Jurek-Vervaat} Jurek, Z. J. and W. Vervaat (1983). An Integral
Representation for Selfdecomposable Banach Space Valued Random Variables.
\emph{Z. Wahrscheinlichkeitstheorie verw. Gebiete} \textbf{62}, 247--262.

\bibitem{KajTaqqu}
Kaj I.  and M. S. Taqqu. (2008).
Convergence to Fractional Brownian Motion
and to the Telecom Process: the Integral Representation Approach.
{\em Progress in Probability}, Vol. 60, 383--427, Birkh\"auser Verlag, Basel.

\bibitem{Kallenberg}
Kallenberg O. (2021). {\em Foundations of Modern Probability.} 3rd Edn, Springer,
New York.

\bibitem{Takis}
Konstantopoulos, T. and S. J. Lin. (1998). Macroscopic models for long-range
dependent network traffic. {\em Queueing Systems Theory Appl.}
{\bf 28}, 215--243.


\bibitem{Lee}
Lee P. M. (1967). Infinitely divisible stochastic processes
{\em Zeitschrift f\"ur Wahrscheinlichkeitstheorie und Verwandte Gebiete},
{\bf 7}, 147--160.


\bibitem{Maruyama}
Maruyama G. (1970). Infinitely divisible processes.
{\em Theory of Probability and its Applications}, {\bf 15} (1), 3--23.

\bibitem{MaulikResnick} Maulik K.  and S. Resnick (2003).
Small and Large Time Scale Analysis of a Network
Traffic Model, {\em Queueing Systems}, {\bf 43}, 221--250.


\bibitem{mikosch}
Mikosch T., S. Resnick, H. Rootz\'en and A. Stegeman, (2002).
Is network traffic approximated by stable L\'evy
motion or fractional brownian motion?
{\em The Annals of Applied Probability,} {\bf 12} (1), 23--68.

\bibitem{MKSA} Mikosch T. and G. Samorodnitsky, (2007).
Scaling Limits for Cumulative Input Processes,
{\em Mathematics of Operations Research,}
{\bf 32} (4), 890-919.

\bibitem{Niculescu} Nicolescu P. and L. Persson (2006).
{\em Convex Functions and Their Applications: A Contemporary Approach.}
Springer.


\bibitem{Pipiras}
Pipiras V. and M. S. Taqqu (2000).
The limit of a renewal reward process with heavy-tailed rewards is
not a linear fractional stable motion. {\em Bernoulli,} {\bf 6}, 4, 2000, 607--614.

\bibitem{ResnickBerg}
Resnick S and E. van den Berg (2000).  Weak convergence of high-speed
network traffic models. {\em Journal of Applied Probability}, {\bf 37},2, 575-597.

\bibitem{ResnickSamorodnitsky}
Resnick, S. and G. Samorodnitsky. (2003).
Limits of ON/OFF hierarchical product models for data transmission
{\em The Annals of Applied Probability,} {\bf 13}, 4, 1355--1398.

\bibitem{Sato} Sato, K. (1999). \emph{L\'evy Processes and Infinitely
Divisible Distributions.} Cambridge University Press, Cambridge.

\bibitem{Shepp1} Shepp L.A. (1971). First passage time for a particular
Gaussian process. {\em The Annals of Mathematical Statistics}, {\bf 42,} (3)
946-951.

\bibitem{Slepian} Slepian D. (1961).
First Passage Time for a Particular Gaussian Process
{\em Annals of Mathematical Statistics,} {\bf 32,} (2) 610-612.

\bibitem{wolfe} Wolfe S.J. (1982). On a continuous analogue of the
stochastic difference equation $X_n = \rho X_{n-1} + B_n$. \emph{Stoch.
Proc. Appl.} \textbf{12}, 301--312.

\bibitem{WolpertTaqqu}
Wolpert R. L. and  M. S. Taqqu. (2005). Fractional Ornstein–Uhlenbeck
L\'evy processes and the Telecom process: Upstairs and downstairs,
{\em Signal Processing,} {\bf 85},  1523--1545.

\end{thebibliography}
\end{document}